\documentclass[10pt,a4paper]{amsart}

\usepackage{amsmath}
\usepackage{amsthm}
\usepackage{enumerate} 
\usepackage{amssymb}
\usepackage{wasysym}
\usepackage{graphicx}
\usepackage[all]{xy}
\usepackage{hyperref}
\usepackage{aliascnt}
\usepackage{stmaryrd} 
\usepackage{mathtools}
\usepackage{verbatim} 
\usepackage{txfonts} 
\usepackage{lmodern}
\usepackage{setspace}   

\newtheorem{lma}{Lemma}[section]

\newaliascnt{thmCt}{lma}
\newtheorem{thm}[thmCt]{Theorem}
\aliascntresetthe{thmCt}

\newaliascnt{corCt}{lma}
\newtheorem{cor}[corCt]{Corollary}
\aliascntresetthe{corCt}

\newaliascnt{propCt}{lma}
\newtheorem{prop}[propCt]{Proposition}
\aliascntresetthe{propCt}

\newtheorem*{thm*}{Theorem}
\newtheorem*{dfn*}{Definition}
\newtheorem*{cor*}{Corollary}
\newtheorem*{rmk*}{Remark}
\newtheorem*{prop*}{Proposition}

\theoremstyle{definition}

\newaliascnt{prgCt}{lma}
\newtheorem{prg}[prgCt]{}
\aliascntresetthe{prgCt}

\newaliascnt{dfnCt}{lma}

\aliascntresetthe{dfnCt}

\newaliascnt{rmkCt}{lma}

\aliascntresetthe{rmkCt}

\newaliascnt{rmksCt}{lma}

\aliascntresetthe{rmksCt}

\newaliascnt{ntnCt}{lma}

\aliascntresetthe{ntnCt}

\newaliascnt{qstCt}{lma}

\aliascntresetthe{qstCt}

\newaliascnt{prblCt}{lma}

\aliascntresetthe{prblCt}

\newaliascnt{exaCt}{lma}

\aliascntresetthe{exaCt}

\newcommand{\T}{\mathbb{T}}
\newcommand{\N}{\mathbb{N}}

\newcommand{\Z}{\mathbb{Z}}
\newcommand{\K}{\mathrm{K}}

\DeclareMathOperator{\im}{im}
\DeclareMathOperator{\rank}{rank}

\DeclareMathOperator{\card}{card}

\DeclareMathOperator{\spectrum}{sp}

\DeclareMathOperator{\her}{her}
\DeclareMathOperator{\Inv}{Inv}

\newcommand{\alg}{\mathrm{alg}}

\newcommand{\CatCa}{C^*}
\newcommand{\CatPoM}{\mathrm{PoM}}
\newcommand{\CatoM}{\mathrm{Mon}_\leq}

\DeclareMathOperator{\sr}{sr}
\DeclareMathOperator{\Gr}{Gr}

\DeclareMathOperator{\Lat}{Lat}

\DeclareMathOperator{\AbGp}{AbGp}
\DeclareMathOperator{\Cu}{Cu}
\DeclareMathOperator{\rr}{rr}

\DeclareMathOperator{\NCCW}{NCCW}
\DeclareMathOperator{\CW}{CW}
\DeclareMathOperator{\AI}{AI}
\DeclareMathOperator{\A}{A}

\DeclareMathOperator{\AH}{AH}

\DeclareMathOperator{\Lsc}{Lsc}

\DeclareMathOperator{\supp}{supp}
\DeclareMathOperator{\Hom}{Hom}
\DeclareMathOperator{\id}{id}

\newcommand{\hooklongrightarrow}{\lhook\joinrel\longrightarrow}

\begin{document}
\onehalfspacing

\title{The unitary Cuntz semigroup on the classification of non-simple $\CatCa$-algebras}

\author{Laurent Cantier}

\address{Laurent Cantier,
Departament de Matem\`{a}tiques \\
Universitat Aut\`{o}noma de Barcelona \\
08193 Bellaterra, Barcelona, Spain
}
\email[]{lcantier@mat.uab.cat}

\thanks{The author was supported by MINECO through the grant BES-2016-077192 and partially supported by the projects MDM-2014-0445 and MTM-2017-83487-POP at the Centre de Recerca Matem\`atica in Barcelona.}

\begin{abstract}
This paper argues that the unitary Cuntz semigroup, introduced in \cite{C21} and termed $\Cu_1$, contains crucial information regarding the classification of non-simple $\CatCa$-algebras.  We exhibit two non-simple $\CatCa$-algebras that agree on their $\K_1$-groups and their Cuntz semigroups (termed $\Cu$) and yet disagree at level of their unitary Cuntz semigroups. In the process, we establish that the unitary Cuntz semigroup contains rigorously more information about non-simple $\CatCa$-algebras than $\Cu$ and $\K_1$ alone.
\end{abstract}
\maketitle

\section{Introduction}
The classification problem for $\CatCa$-algebras, inspired by the results obtained for von Neumann algebras, has been mainly focused on the simple case. Nowadays, through the work of numerous people, gathering decades of research, the Elliott classification program has provided satisfactory results. See e.g. \cite{EGLN21}, \cite{GLN1} and \cite{GLN2} More specifically, the original Elliott invariant constructed from $\K$-Theory and traces, is a complete invariant for simple, separable, unital, nuclear, $\mathcal{Z}$-stable $\CatCa$-algebras satisfying the UCT assumption. Besides, this invariant has been the foundation for classification of certain non-simple $\CatCa$-algebras. By adding up new ingredients, Gong, Jiang and Li (based on the work of many others) have reformulated more complete versions of the Elliott invariant in the aim of classifying $\CatCa$-algebras beyond the simple case. (See e.g. \cite{JJ11}, \cite{GJL20}.) Among these new components, we may find the affine maps from the tracial space of corner algebras (with compatibility axioms) and we may observe that the total $\K$-Theory comes to replace the original $\K_*$-group. A first version obtained, denoted by $\Inv^{0}$, partially classifies $\AH$-algebras with ideal property and no dimension growth, a large class of $\CatCa$-algebras that contains both simple and real rank zero $\AH$-algebras. Later, in \cite{GJL19} and \cite{GJL20}, they round off the classification of all $\AH$-algebras with ideal property and no dimension growth by further refining of their invariant, now termed $\Inv$, by adding Hausdorffified $\K_1$-groups of corner algebras (with some more compatibility axioms). 

In the meantime, it has been established that the Cuntz semigroup, introduced in \cite{C78} as an analogue of the Murray von-Neumann semigroup for $\CatCa$-algebras, encodes a significant amount of information about a $\CatCa$-algebra. This semigroup, built upon equivalent classes of positive elements, did not receive the appropriate attention it deserved back in the late 70's: its computation is rather complex (in opposition to $\K$-theoretical invariants) and the original definition did not preserve inductive limits. However, in the late 00's, Toms exhibited a counter-example to the original Elliott conjecture (see \cite{T08}) where the extra data needed to distinguish the $\CatCa$-algebras was encapsulated in the Cuntz semigroup. Since then, a \textquoteleft completed version\textquoteright\ (termed $\Cu$) has been introduced in \cite{CEI08} and has been extensively studied since. Currently, it is well-established that the Cuntz semigroup is a powerful tool for classification of $\CatCa$-algebras (beyond the simple case) and the assets it contains have not been fully unraveled yet. For instance, it has been shown in \cite{ADPS14} that the Cuntz semigroup of the tensor product of $\mathcal{C}(\T)$ with any unital, simple, separable, nuclear, finite, $\mathcal{Z}$-stable $\CatCa$-algebra $A$ is naturally isomorphic to the original Elliott invariant of $A$. Furthermore, the Cuntz semigroup entirely captures the lattice of ideals of any separable $\CatCa$-algebra $A$ through the assignment $I\longmapsto \Cu(I)$ which defines a complete lattice isomorphism from the complete lattice of ideals of $A$ into the complete lattice of $\Cu$-ideals of $\Cu(A)$. Lastly, Robert was able to classify all the unital one-dimensional $\NCCW$ complexes with trivial $\K_1$-group and inductive limits of such building blocks by means of the functor $\Cu$. (See \cite{R12}.) However, it is expected that the Cuntz semigroup cannot capture $\K_1$-related information. To overcome obstacle, the author introduced in \cite{C21}, a unitary version of the Cuntz semigroup, written $\Cu_1$, for separable $\CatCa$-algebras with stable rank one. Since, the author has been intensively studying functorial properties, computations, ideals and exactness properties of this new invariant that are reminded in what follows. We refer the reader to \cite{C21} and \cite{C21b} for more details. 

This paper points out extra data encoded in the unitary version of the Cuntz semigroup that cannot be retrieve from the Cuntz semigroup and the $\K_1$-group alone. A fortiori, we show that the unitary Cuntz semigroup is a promising candidate for classification of $\CatCa$-algebras of stable rank one, outside the simple case and beyond trivial $\K_1$ case. For that matter, we exhibit two non-simple unital separable $\CatCa$-algebras $A$ and $B$ that are not isomorphic. $\K$-Theory and the Cuntz semigroup cannot distinguish these $\CatCa$-algebras, and yet additional information contained in $\Cu_1$ is enough to show that $A\nsimeq B$. \\

\textbf{Organization of the paper.} In a first part, we recall the definition and main properties of the Cuntz semigroup and its unitary version. We also recall a metric for $\Cu$-morphisms that has been introduced by the author in \cite{C22}. We finally recall the construction of Evans-Kishimoto folding interval algebras that will play a key-role in the proof of the main theorem of the paper that we state below.
\begin{thm*}
There exist two non-simple, separable, unital $\AH$-algebras of stable rank one $A$ and $B$ such that 
\[
\left\{
\begin{array}{ll} 
\Cu(A)\simeq \Cu(B).\\
\K_1(A)\simeq \K_1(B).
\end{array}
 \right.
\vspace{0,3cm}\] 
Yet $A$ and $B$ are not isomorphic since they are distinguished by their unitary Cuntz semigroup. Moreover, both $A$ and $B$ are inductive limits of $1$-dimensional $\NCCW$-complexes. 
\end{thm*}
\break

\textbf{Acknowledgements.} The author is indebted to R. Antoine and F. Perera for pertinent comments through the elaboration of this paper. More particularly, for fixing an issue in the computation of $\Cu$-distances between the two inductive sequences. Further, the author would like to thank G. Elliott, J. Gabe and A. Tikuisis for feedbacks that led to the remarks ending this paper. Finally, the author would like to thank the referee for his/her thorough proofreading.
\section{Preliminaries}

We use $\AbGp$ to denote the category of abelian groups. We use $\CatoM$ to denote the category of ordered monoids, in contrast to the category of positively ordered monoids, that we write $\CatPoM$. We also use $\CatCa_{\sr1}$ to denote the full subcategory of $\CatCa$-algebras of stable rank one.

\subsection{The Cuntz semigroup and its unitary version}
The unitary Cuntz semigroup is an enhancement of the original Cuntz semigroup. More precisely, the Cuntz semigroup being, roughly speaking, equivalence classes of positive elements, its unitary version incorporates additional ingredients from the $\K_1$-group, via unitary elements. In the sequel, we recall both constructions of the Cuntz semigroup and its unitary version, as well as basic properties of the two latter. We refer the reader to \cite{APS11}, \cite{APT14}, \cite{APRT21}, \cite{APT09}, \cite{C22}, \cite{CEI08}, \cite{RS19}, \cite{T19} and \cite{TV21} for the original Cuntz semigroup and to \cite{C21} and \cite{C21b} for the unitary Cuntz semigroup. Even though the original Cuntz semigroup is not needed the stable rank one assumption, its unitary version does. For that matter,  all $\CatCa$-algebras are assumed to have stable rank one. 

\begin{prg}\textbf{(Construction of $\Cu(A)$ and $\Cu_1(A)$.)}
Let $A$ be a $\CatCa$-algebra of stable rank one. We denote by $A_+$ the set of positive elements. Let $a$ and $b$ be in $A_+$. We say that $a$ is \emph{Cuntz subequivalent} to $b$ and we write $a\lesssim_{\Cu} b$ if, there exists a sequence $(x_n)_{n\in\N}$ in $A$ such that $a=\lim\limits_{n\in\N}x_nbx_n^*$. Furthermore, in the stable rank one context, whenever $a\lesssim_{\Cu} b$, there exist \emph{standard morphisms} $\theta_{ab}:\her (a)^\sim\hooklongrightarrow \her (b)^\sim$ that all give rise to the same $\K_1$-morphism $\chi_{ab}:=\K_1(\theta_{ab})$. In other words, there is a canonical way (up to homotopy equivalence) to extend unitary elements of $\her (a)^\sim$ into unitary elements of $\her (b)^\sim$. (See \cite[2.2-3.2]{C21} for details.) Now, let $u,v$ be unitary elements of $\her (a)^\sim,\her (b)^\sim$ respectively. We say that $(a,u)$ is \emph{unitarily Cuntz subequivalent} to $(b,v)$, and we write $(a,u)\lesssim_1 (b,v)$, if $a\lesssim_{\Cu} b$ and $\theta_{ab}(u)\sim_h v$. After antisymmetrizing these relations, we get equivalence relations on the sets $(A\otimes\mathcal{K})_+$ and $H(A):=\{(a,u)\mid a\in (A\otimes\mathcal{K})_+ , u\in \mathcal{U}(\her (a)^\sim)\}$, called the \emph{Cuntz equivalence} and the \emph{unitary Cuntz equivalence}, and denoted by $\sim_{\Cu}$ and $\sim_1$ respectively. 
We now construct \[
\Cu(A):= (A\otimes\mathcal{K})_+/\!\!\sim_{\Cu} \text{ and } \Cu_1(A):= H(A)/\!\!\sim_{1}.\] The set $\Cu_1(A)$ can be equipped with a natural order given by $[(a,u)]\leq [(b,v)]$ whenever $(a,u)\lesssim_{1} (b,v)$. Further, let us consider the $^*$-isomorphism $\psi: M_2(A\otimes\mathcal{K})\longrightarrow A\otimes\mathcal{K}$ given by $\psi(\begin{smallmatrix} a & 0\\ 0 & b \end{smallmatrix})=v_1av_1^*+v_2bv_2^*$, where $v_1$ and $v_2$ are two isometries in the multiplier algebra of $A\otimes\mathcal{K}$ such that $v_1v_1^*+v_2v_2^*=1_{M(A\otimes\mathcal{K})}$ and let us write $x \oplus y:=\psi(\begin{smallmatrix} x & 0\\ 0 & y \end{smallmatrix})$. This allows us to set an addition on $\Cu_1(A)$ by $[(a,u)]+[(b,v)]:=[(a\oplus b,u\oplus v)]$. 

In this way, $\Cu_1(A)$ is an ordered monoid called \emph{the unitary Cuntz semigroup of $A$}. (A fortiori, the set $\Cu(A)$ can also be equipped with a natural order and an addition to obtain a positively ordered monoid called the \emph{Cuntz semigroup}.)

Finally, any $^*$-homomorphism $\phi:A\longrightarrow B$ naturally induces a semigroup morphism $\Cu_1(\phi):\Cu_1(A)\longrightarrow\Cu_1(B)$, by sending $[(a,u)]\longmapsto [(\phi\otimes id_\mathcal{K})(a),(\phi\otimes id_\mathcal{K})^\sim(u)]$. 

Therefore, we get a functor $\Cu_1$ from the category $\CatCa_{\sr1}$ into a certain subcategory of $\CatoM$ called the \emph{unitary Cuntz category}, and written $\Cu^\sim$, that we describe next. (Similarly, $\phi$ induces a semigroup morphism $\Cu(\phi):\Cu(A)\longrightarrow \Cu(B)$ and hence are able to define a functor $\Cu$ from the category of $\CatCa$-algebras into a certain subcategory of $\CatPoM$ called the \emph{Cuntz category}, and written $\Cu$.) 
\end{prg}

\begin{prg}\textbf{(The categories $\Cu$ and $\Cu^\sim$.)}
\label{dfn:llCU}
The categories $\Cu$ and $\Cu^\sim$ are subcategories of $\CatPoM$ and $\CatoM$ respectively. As implicitly stated, the category $\Cu$ is a full subcategory of $\Cu^\sim$ consisting of objects in $\Cu^\sim$ that are positively ordered. We hence only recall the definition of the category $\Cu^\sim$. 

Let $(S,\leq)$ be an ordered monoid and let $x,y$ in $S$. We say that $x$ is \emph{way-below} $y$ and we write $x\ll y$ if, for all increasing sequences $(z_n)_{n\in\N}$ in $S$ that have a supremum, if $\sup\limits_{n\in\N} z_n\geq y$, then there exists $k$ such that $z_k\geq x$. This is an auxiliary relation on $S$ called the \emph{way-below relation} or the \emph{compact-containment relation}. In particular $x\ll y$ implies $x\leq y$ and we say that $x$ is a \emph{compact element} whenever $x\ll x$. 

We say that $S$ is an \emph{abstract unitary Cuntz semigroup}, or a $\Cu^\sim$-semigroup, if $S$ satisfies the following order-theoretic axioms: 

(O0): $0\ll 0$.

(O1): Every increasing sequence of elements in $S$ has a supremum. 

(O2): For any $x\in S$, there exists a $\ll$-increasing sequence $(x_n)_{n\in\N}$ in $S$ such that $\sup\limits_{n\in\N} x_n= x$.

(O3): Addition and the compact containment relation are compatible.

(O4): Addition and suprema of increasing sequences are compatible.

A \emph{$\Cu^\sim$-morphism} between two $\Cu^\sim$-semigroups $S,T$ is an ordered monoid morphism that preserves the compact containment relation and suprema of increasing sequences. 

The category of abstract unitary Cuntz semigroups, written $\Cu^\sim$, is the subcategory of $\CatoM$ whose objects are $\Cu^\sim$-semigroups and morphisms are $\Cu^\sim$-morphisms. It is proven in \cite[Corollary 3.21]{C21}, that the functor $\Cu_1$ from the category $\CatCa_{\sr1}$ to the category $\Cu^\sim$ is arbitrarily continuous.
\end{prg}

\begin{prg}\textbf{(Alternative picture of the $\Cu_1$-semigroup.)}
\label{prg:newpicture}
It may be useful to consider an alternative picture of the unitary Cuntz semigroup.
First, recall that for a $\CatCa$-algebra $A$, $\Lat_f(A)$ is the sublattice of $\Lat(A)$ consisting of ideals that contain a full, positive element. Also recall that $\{\sigma\text{-unital ideals of }A\}\subseteq \Lat_f(A)$ and if moreover $A$ is separable, then the converse inclusion holds. Finally, for any $I\in\Lat_f(A)$, we define $\Cu_f(I):=\{x\in \Cu(A) \mid I_x=\Cu(I)\}$ to be the set of full elements in $\Cu(I)$. 

Let $A$ be a $\CatCa$-algebra of stable rank one such that $\Lat_f(A)=\{\sigma\text{-unital ideals of }A\}$.
Then $\Cu_1(A)$ can be pictured as \[\bigsqcup\limits_{I\in\Lat_f(A)} \Cu_f(I)\times \K_1(I)\] that we also write $\Cu_1(A)$. The addition and order are defined as follows: For any $(x,k),(y,l)\in \Cu_1(A)$
\[
\left\{
\begin{array}{ll}
(x,k)\leq (y,l) \text{ if, } x\leq y \text{ and } \delta_{I_xI_{y}}(k)=l.\\
(x,k)+(y,l):=(x+y,\delta_{I_xI_{x+y}}(k)+\delta_{I_yI_{x+y}}(l)).
\end{array}
\right.
\]
where $\delta_{IJ}:=\K_1(I\overset{i}\hooklongrightarrow J)$, for any $I,J\in\Lat_f(A)$ such that $I\subseteq J$.

Let $A,B$ be $\CatCa$-algebras of stable rank one and let $\phi:A\longrightarrow B$ be a $^*$-homomorphism. For any $I\in\Lat_f(A)$, we write $J:=\overline{B\phi(I)B}$, the smallest ideal of $B$ that contains $\phi(I)$. Then $J\in\Lat_f(B)$ and $\Cu_1(\phi)$ can be rewritten as $(\Cu(\phi),\{\K_1(\phi_{|I})\}_{I\in\Lat_f(A)})$, where $\phi_{|I}:I\longrightarrow J$. Observe that we might write $\alpha,\alpha_0, \alpha_I$ to denote $\Cu_1(\phi),\Cu(\phi),\K_1(\phi_{|I})$ respectively.

Note that this way, it is trivial to see that for any (separable) simple $\CatCa$-algebra $A$, we have  \[\Cu_1(A)\simeq \{0\}\sqcup (\Cu(A)\setminus\{0\})\times\K_1(A).\]
\end{prg}

\begin{prg}\textbf{(Positive elements and maximal elements of the $\Cu_1$-semigroup.)} Let $A$ be a $\CatCa$-algebra of stable rank one. It is naturally expected to be able to recover $\Cu(A)$ and $\K_1(A)$ from $\Cu_1(A)$. It has been shown in \cite[3.4 and 5.1]{C21} that there exist functors $\nu_+:\Cu^\sim\longrightarrow \Cu$ and $\nu_{\max}:\Cu^\sim\longrightarrow \AbGp$ that respectively assign to an abstract unitary Cuntz semigroup $S$, its positive cone $S_+:=\{x\in S\mid x\geq 0\}$ and the set of its maximal elements $S_{\max}:=\{x\in S\mid \text{if } y\geq x \text{ then } y=x\}$, satisfying the following natural isomorphisms:
\[
\nu_+\circ\Cu_1\simeq \Cu \hspace{2cm} \nu_{\max}\circ\Cu_1\simeq \K_1.
\]
Moreover, it has been shown in \cite[Theorem 4.13]{C21b} that these objects, considered as $\Cu^\sim$-semigroups, are linked in the following split-exact sequence:
\[
\xymatrix{
0\ar[r]^{} & S_+\ar[r]^{} & S\ar[r]^{} & S_{max}\ar@{.>}@/_{-1pc}/[l]^{}\ar[r]^{} & 0
} 
\]
which leads to 
\[
\xymatrix{
0\ar[r]^{} & \Cu(A)\ar[r]^{} & \Cu_1(A)\ar[r]^{} & \K_1(A)\ar[r]^{} & 0.
} 
\]
\end{prg}

\begin{prg}\textbf{(Ideals in $\Cu$ and $\Cu^\sim$.)} Let $S$ be a $\Cu$-semigroup. An \emph{ideal} of $S$ is a submonoid $I$ that is also a lower set (in the sense that, for any $x,y\in S$ with $x\leq y$ and $y\in I$, then $x\in I$) which is closed under suprema of increasing sequences. We mention that the two latter properties define the closed sets of $S$ for the \emph{Scott topology}, see \cite[Definition 3.2]{C21b}. Furthermore, the ideal $I_x$ generated by an element $x\in S$ (defined as the smallest ideal of $S$ containing $x$) always exists and we have that $I_x=\{y\in S \mid y\leq \infty x\}$.
We denote the set of $\Cu$-ideals of $S$ by $\Lat(S\!)$ and we set $I\wedge J:= I\cap J$ and $I\vee J:=I+J$. It has been shown in \cite[\S 5.1.6]{APT14} that $\Lat(S\!)$ is a complete lattice under these  operations. 

 In a more general context where $S$ is a $\Cu^\sim$-semigroup, things get slightly more complex as a result of the fact that $S$ need not be positively ordered. An \emph{ideal} of $S$ is a submonoid $I$ that is not only Scott-closed but also \emph{positively stable}. (See \cite[Definition 3.9]{C21b} for more details.) Naturally, both definitions agree for $\Cu$-semigroups. Note that the intersection of two $\Cu^\sim$-ideals might not be an ideal and hence, the ideal generated by an element (defined as the smallest ideal of $S$ containing $x$) might not always exist. However, the ideal $I_x$ generated by a \emph{positive} element $x\in S_+$ always does and we have that $I_x=\{y\in S \mid \text{ there is } y'\in S \text{ with } 0\leq y+y'\leq\infty x\}$. 
 
Finally, for any $\CatCa$-algebra of stable rank one $A$, we have that $\Cu_1(I)$ is an ideal of $\Cu_1(A)$ for any $I\in\Lat(A)$. (And $\Cu(I)$ is an ideal of $\Cu(A)$.) In fact, we have two lattice isomorphisms as follows:
\[
	\begin{array}{ll}
		\Lat(A)\overset{\simeq}{\longrightarrow} \Lat(\Cu_1(A))\overset{\simeq}{\longrightarrow} \Lat(\Cu(A))\\
		\hspace{0,9cm} I\longmapsto\hspace{0,4cm} \Cu_1(I)\hspace{0,5cm}\longmapsto \Cu_1(I)_+
	\end{array}
\]
We refer the reader to \cite{APT14} and \cite{C21b} for more details on $\Cu$/$\Cu^\sim$-ideals.
\end{prg}

\subsection{Uniformly Based Cuntz semigroups and \texorpdfstring{$\Cu$}{Cu}-metrics}\label{sec:uniformbasis} The author has introduced in \cite{C22} a notion of \emph{uniform basis} for an abstract Cuntz semigroup $S$. In a nutshell, a uniform basis allows us to approach any element $s$ of $S$ by an increasing sequence $(\epsilon_n(s))_n$ in $s_\ll:=\{x\in S\mid x\ll s\}$ in a \emph{uniform way}. In the sense that, for any $n\in\N$, there exists a positively ordered monoid $M_n$ such that $\epsilon_n(S)\subseteq M_n$. (See \cite[Definition 3.2]{C22}.) 

As for examples, the countable set $(\N,\id_\N)_n$ is a uniform basis for the $\Cu$-semigroup $\overline{\N}$. Moreover, it has been shown that any $\Cu$-semigroup of the form $\Lsc(X,\overline{\N}^r)$, where $X$ is a one-dimensional compact $\CW$-complex and $r\in\N$, has a uniform basis. More specifically, for any supernatural number $w:=\overset{\infty}{\underset{i=0}{\Pi}}\, p_i$, there exists an associated uniform basis $\mathcal{B}_{w}$ which relies upon increasingly finer closed (finite) covers $\mathcal{U}_n:=\{\overline{U_{k}}\}_{1}^{w_n}$, where $w_n:=\overset{n}{\underset{i=0}{\Pi}}\, p_i$, induced by increasingly finer equidistant partitions $(x_{k})_{0}^{w_n}$ of $X$ of size $1/w_n$. We refer the reader to \cite[\S4.A and \S4.B]{C22} for an explicit construction. We may use these finite closed covers of $X$ induced by equidistant partitions of any finite size $1/w$, and we will denote such a cover by $\{\overline{U_k}\}_1^w$.

Further, a uniform basis for a $\Cu$-semigroup $S$ yields a (semi)-metric on the set $\Hom_{\Cu}(S,T)$, where $T$ is any $\Cu$-semigroup. Let $S$ be a uniformly based $\Cu$-semigroup and let $\mathcal{B}:=(M_n,\epsilon_n)_n$ be a uniform basis of $S$. Let $\alpha,\beta:S\longrightarrow T$ be two $\Cu$-morphisms, where $T$ is any $\Cu$-semigroup. We define
\[
dd_{\Cu,\mathcal{B}}(\alpha,\beta):= \inf\limits_{n\in\N} \{ \frac{1}{n} \mid \alpha\underset{M_n}{\simeq}\beta \}
\]  
where $\alpha\underset{M_n}{\simeq}\beta$ if, for any $g'\ll g$ in $M_n$ we have that $\alpha(g')\leq \beta(g)$ and $\beta(g')\leq \alpha(g)$. (See \cite[Definition 3.8]{C22}). If the infimum does not exist, we set the value to $\infty$. This defines a semi-metric on $\Hom_{\Cu}(S,T)$. We mention that whenever $S$ is of the above form (that is, $S=\Lsc(X,\overline{\N}^r)$), any two bases $\mathcal{B}_{w},\mathcal{B}_{w'}$ obtained from any two supernatural numbers $w,w'$ induce topologically equivalent semi-metrics.

Finally, the author obtained an approximate intertwining theorem for uniformly based $\Cu$-semi\-groups that we recall next, when we prove that the Cuntz semigroups of the $\CatCa$-algebras that we construct (as inductive limits) are isomorphic. 

\subsection{Evans-Kishimoto folding interval algebras}\label{sec:EK} These $\CatCa$-algebras, constructed as pullbacks of interval algebras and finite dimensional $\CatCa$-algebras, can be seen as a generalization of the well-known Elliott-Thomsen dimension drop interval algebras. That being said, they have in fact been considered in \cite{EK91} a few years before the dimension drop algebras. 
We recall their construction and expose some of their properties. (We refer the reader to \cite[Section 2]{EK91} and more precisely to \cite[(2.14)]{EK91} for the original construction.)

Let $q$ be a natural number and denote the full matrix algebra of size $q$ by $M_q$. Let $e$ be a projection of $M_q$ and write $p:=\rank(e)$. For any $l\in\N_*$, we consider the following pullback:
\[
\xymatrix{
\mathcal{I}^l_{q,e}\ar[r]^{\pi_1}\ar[d]_{\pi_2}& C([0,1],\underset{1}{\overset{l}\otimes} M_q)\ar[d]^{(ev_0, ev_1)}\\
(\underset{1}{\overset{l-1}\otimes} M_q)\oplus (\underset{1}{\overset{l-1}\otimes} M_q)\ar[r]_{(i_0^l, i_1^l)}& (\underset{1}{\overset{l}\otimes} M_q)\oplus (\underset{1}{\overset{l}\otimes} M_q)
}
\]
where $i_0^l,i_1^l: (\underset{1}{\overset{l-1}\otimes} M_q)\oplus (\underset{1}{\overset{l-1}\otimes} M_q)\longrightarrow \underset{1}{\overset{l}\otimes} M_q$ are injective $^*$-homomorphisms constructed by induction. We refer the reader to  \cite[(2.11)/(2.12)]{EK91} for more details and we mention that the original construction involves two projections $E_1$ and $E_2$, that are given in our case by $E_1:=e$ and $E_2:= 1-e$. 
Let us precise that for $l=1$, we define $\underset{1}{\overset{0}\otimes} M_q:=\mathbb{C}$ and we exactly recover the Elliott-Thomsen dimension-drop interval algebra of size $q$. 

Let $l\in\N_*$. Using the 6-term exact sequence, we get that
\[
\K_0(\mathcal{I}^l_{q,e})\simeq \Z\hspace{3cm} \K_1(\mathcal{I}^l_{q,e})\simeq \Z/q\Z.\]
Now, arguing similarly as in \cite[Proof of 2.2]{EK91}, we consider the paths $\xi_0:t\longmapsto t/2$ and $\xi_1:t\longmapsto 1-t/2$ in $[0,1]$. They yield the following $^*$-homomorphism:
\[
\begin{array}{ll}
\psi_{l,e}: \mathcal{I}_{q,e}^l\longrightarrow \mathcal{I}_{q,e}^{l+1}\\
\hspace{1,2cm}f\longmapsto f(\xi_0)\otimes e+ f(\xi_1)\otimes (1-e)
\end{array}
\]
At level of $\K$-theory, we obtain that 
\[
	\begin{array}{ll}
		\K_0(\psi_{l,e}):\Z\overset{\times q}\longrightarrow\Z \hspace{3cm}
		\K_1(\psi_{l,e}):\Z/q\Z\overset{\times \rank(e)}\longrightarrow\Z/q\Z.	\end{array}
\]
Finally, the Cuntz semigroup of $\mathcal{I}_{q,e}^l$ has been studied in \cite[\S 4.20]{C22}. First, it can be computed as follows: \begin{align*}	 
\Cu(\mathcal{I}^l_{q,e})&\simeq \{f\in \Lsc([0,1],\overline{\N}) \mid f(0),f(1)\in q\overline{\N}\}\\
	&\simeq \{f\in \Lsc([0,1],\frac{1}{q^l}\overline{\N}) \mid f(0), f(1)\in \frac{q}{q^l}\overline{\N}\}.
\end{align*}
Moreover, it is shown that $\Cu(\mathcal{I}^l_{q,e})$ is uniformly based. More precisely, for any supernatural number $w:=\overset{\infty}{\underset{i=0}{\Pi}}\, p_i$, the associated uniform basis $\mathcal{B}_w=(M_n,\epsilon_n)_n$ of $\Lsc([0,1],\overline{\N})$ of size $w$ induces a uniform basis $\mathcal{B}'_w:=(M'_n,\epsilon_n')_n$ of $\Cu(\mathcal{I}^l_{q,e})$ defined by
\[
\begin{array}{ll} 
\left\{
\begin{array}{ll} 
\vspace{0,1cm}M'_n:=\{f\in\Lsc(X,\frac{1}{q^l}\N)\mid f_{|U_{k}} \text{ is constant for any }k\in \{1,\dots,w_n\} \text{ and } f(0), f(1)\in \frac{q}{q^l}\overline{\N}\}.\\ 
\epsilon'_n: \{f\in\Lsc(X,\frac{1}{q^l}\N)\mid f(0), f(1)\in \frac{q}{q^l}\overline{\N}\}\longrightarrow M'_n
 \end{array}
 \right.\\
  \hspace{6,5cm}f\longmapsto\max\limits_{g\in M'_n}\{ g\ll f\}
   \end{array}
\]
where $w_n:=\overset{n}{\underset{i=0}{\Pi}}\, p_i$ and $\{\overline{U_{k}}\}_{k=1}^{w_n}$ is the finite closed cover induced by the (unique) equidistant partition of the interval of size $1/w_n$. 
Naturally, we refer to $\mathcal{B}'_w$ as \emph{the uniform basis of $\Cu(\mathcal{I}^l_{q,e})$ of size $w$}. We refer the reader to \cite[\S 4.20]{C22} for more details.

In the sequel, we shall use the uniform basis of $\Cu(\mathcal{I}^l_{q,e})$ of size $2^\infty$ and its associated semi-metric $dd_{\Cu,\mathcal{B}'_{2^\infty}}$ defined on $\Hom_{\Cu}(\Cu(\mathcal{I}^l_{q,e}),T)$ as follows: 
\[
dd_{\Cu,\mathcal{B}'_{2^\infty}}(\alpha,\beta):= \inf\limits_{n\in\N} \{ \frac{1}{2^n} \mid \alpha\underset{\Lambda_n}{\simeq}\beta \}
\]
where $\Lambda_n:=M'_n\cap\{f\in \Lsc([0,1],\{\frac{j}{q^l}\}_{j=0}^{q}) \mid f(0), f(1)\in \{0,\frac{q}{q^l}\}\}$. (See \cite[\S 5.6]{C22}.)

\section{Proof of the main theorem}

In the aim of proving the result of the paper, we are building two $\CatCa$-algebras $A$ and $B$ as inductive limits of one-dimensional $\NCCW$-complexes. More precisely, we shall construct two inductive sequences of finite direct sums whose building blocks are  matrices of some suitable order over Evans-Kishimoto folding interval algebras. We shall use evaluation maps and $^*$-homomorphisms described earlier (see \autoref{sec:EK}) between the building blocks. Subsequently, we check that these two $\CatCa$-algebras are non-simple, separable, unital $\AH$-algebras of stable rank one and we study the set of (simple) ideals of both $A$ and $B$. This will allow us to get an explicit description of their $\K$-Theory and prove that they agree at level of $\K_*:=\K_0\oplus\K_1$. The next step is to compute the distance between the $\Cu$-morphisms induced by the $^*$-homomorphisms of the inductive systems and conclude that there exists a two-sided approximate intertwining at level of $\Cu$ to prove that $A$ and $B$ agree at level of their Cuntz semigroups. Finally, using once again the set of simple ideals of both algebras, we will be enabled to prove that $\Cu_1(A)\nsimeq\Cu_1(B)$ and hence that $A\nsimeq B$.

The \textquoteleft triangular shape\textquoteright\ of the inductive sequences was greatly inspired by construction done in \cite{GJL19}. Nevertheless the building blocks, properties and arguments involved are quite distinct, so are the proofs.
\subsection{Construction of the blocks}\label{sec:block} Let us define the following sequences

$\bullet$ $(d_k)_k$ is a countable dense subset of $[0,1]$.

$\bullet$ $(p_k)_k$ denotes all the prime numbers in increasing order. (By convention, $p_0=2$ and $p_{-1}=1$.)

$\bullet$ $(q_k)_k$ is defined by $q_k:=p_kp_{k-1}$. (Thus, $q_0=2$.)

$\bullet$ $(r_k)_k$ is a strictly increasing sequence of natural numbers such that $r_0\geq 2$.

For any $k\in\N$, let $e_A^k,e_B^k$ be projections of $M_{q_k}$ of ranks $p_{k-1},p_k$ respectively. Now, for any $l\in\N$, we define the following folding interval algebras:
\[
\left\{
    \begin{array}{ll}
    		\mathcal{I}_{q_k}^l:=\mathcal{I}_{q_k,e^k_A}^l\\
    		\mathcal{J}_{q_k}^l:=\mathcal{I}_{q_k,e^k_B}^l
    \end{array}
\right.
\]
and we omit $e_A^k,e_B^k$ for notational purposes. We now build the $\CatCa$-algebras of the inductive system for $A$ using matrix algebras over $\mathcal{I}_{q_k}^n$. (The $\CatCa$-algebras of the inductive system for $B$ are constructing similarly, replacing $\mathcal{I}_{q_k}^n$ by $\mathcal{J}_{q_k}^n$.)
\begin{align*}
&A_0 = \mathcal{C}([0,1])\\
&A_1 = M_{q_0^{r_0}}(\mathcal{I}_{q_0}^1) \oplus \mathcal{C}([0,1])\\
&A_2 = M_{q_0^{r_0}q_0^{r_1}}(\mathcal{I}_{q_0}^2) \oplus M_{q_1^{r_1}}(\mathcal{I}_{q_1}^1) \oplus \mathcal{C}([0,1])\\
&\hspace{0,2cm}\vdots\\
&A_n = M_{q_0^{r_0}... q_0^{r_{n-1}}}(\mathcal{I}_{q_0}^n) \oplus  ...\oplus M_{q_i^{r_i}... q_i^{r_{n-1}}}(\mathcal{I}_{q_i}^{n-i}) \oplus... \oplus M_{q_{n-1}^{r_{n-1}}}(\mathcal{I}_{q_{n-1}}^{1})\oplus \mathcal{C}([0,1])\\
&\hspace{0,2cm}\vdots
\end{align*}
Let us write $[n,i]:=\prod\limits_{j=i}^{n-1}q_i^{r_j}$, for any $0\leq i\leq n-1$. We finally rewrite
\[
\hspace{-2,1cm}\left\{
    \begin{array}{ll}
    		A_n:=\bigoplus\limits_{i=0}^{n-1}M_{[n,i]}(\mathcal{I}_{q_i}^{n-i})\oplus \mathcal{C}([0,1]).\\
    		B_n:=\bigoplus\limits_{i=0}^{n-1}M_{[n,i]}(\mathcal{J}_{q_i}^{n-i})\oplus \mathcal{C}([0,1]).
    \end{array}
\right.
\]
Notice that $[n+1,i]=q_i^{r_n}[n,i]$, for any $0\leq i\leq n-1$. 
\subsection{Construction of the morphisms}\label{sec:morph} Now that have the $\CatCa$-algebras of the two inductive systems, we construct the morphisms. Observe that in both systems, the $n$-th $\CatCa$-algebra is a direct sum of $n$ matrix algebras over folding interval algebras of size $\{q_i\}_0^{n-1}$ together with $\mathcal{C}([0,1])$. In other words, the $n$-th $\CatCa$-algebra is a direct sum of $n+1$ blocks. The $n$-th $^*$-homomorphism will be constructed using $n+2$ partial $^*$-homomorphisms. 

More precisely, there will be $n$ partial $^*$-homomorphisms from the $n$ matrix algebras (of order $[n,i]$) over folding interval algebras of size $\{q_i\}_0^{n-1}$ into the $n$ matrix algebras (of order $[n+1,i]$) over the same folding interval algebras (that is, of size $\{q_i\}_0^{n-1}$). Futher, there will be $2$ partial $^*$-homomorphisms from $\mathcal{C}([0,1])$: the first one has codomain the additional block the $(n+1)$-th $\CatCa$-algebra (which is, a matrix algebra over folding interval algebra of size $q_n$) while the second one has codomain $\mathcal{C}([0,1])$ and is an evaluation map at $d_n$. 

Let $n\in\N$. We denote $\phi_{nn+1}: A_n\longrightarrow A_{n+1}$ and $\psi_{nn+1}: B_n\longrightarrow B_{n+1}$ and we construct them as follows:\\
$\bullet$ \emph{The $n$ first partial $^*$-homomorphisms for $\phi_{nn+1}$}. Let $0\leq i\leq n-1$, we define
\[
\begin{array}{ll}
	\hspace{0,7cm}\phi_{nn+1}^i: M_{[n,i]}(\mathcal{I}_{q_i}^{n-i})\longrightarrow M_{[n+1,i]}(\mathcal{I}_{q_i}^{n-i+1})\\\\
\hspace{3,58cm} f\longmapsto 
\left(\begin{smallmatrix}
	f(\xi_0)\otimes e_A^i+ f(\xi_1)\otimes (1-e_A^i) &&&&\\
   &f(1/q_i^{r_n})\otimes1_{M_{q_i}}&&&\\
   &&...&&\\ 										&&&& f((q_i^{r_n}-1)/q_i^{r_n})\otimes1_{M_{q_i}}
\end{smallmatrix}\right)
\end{array}
\]
$\bullet$ \emph{The $n$ first partial $^*$-homomorphisms for $\psi_{nn+1}$}. Let $0\leq i\leq n-1$, we define
\[
\begin{array}{ll}
	\hspace{0,7cm}\psi_{nn+1}^i: M_{[n,i]}(\mathcal{J}_{q_i}^{n-i})\longrightarrow M_{[n+1,i]}(\mathcal{J}_{q_i}^{n-i+1})\\\\
\hspace{3,68cm} f\longmapsto 
\left(\begin{smallmatrix}
	 f(\xi_0)\otimes e_B^i+ f(\xi_1)\otimes (1-e_B^i) &&&&\\
   & f(1/q_i^{r_n})\otimes1_{M_{q_i}}&&&\\
   &&...&&\\ 										&&&& f((q_i^{r_n}-1)/q_i^{r_n})\otimes1_{M_{q_i}}
\end{smallmatrix}\right)
\end{array}
\]
$\bullet$ \emph{The $(n+1)$-th partial $^*$-homomorphism for $\phi_{nn+1}$}. We define
\[
\begin{array}{ll}
\phi_{nn+1}^n: \mathcal{C}([0,1])\longrightarrow M_{q_n^{r_n}}(\mathcal{I}_{q_n}^{1}) \\
\hspace{2,22cm} f\longmapsto f(0)\otimes 1_{M_{q_n^{r_n+1}}}
\end{array}
\]
$\bullet$ \emph{The $(n+1)$-th partial $^*$-homomorphism for $\psi_{nn+1}$}. We define
\[
\begin{array}{ll}
\psi_{nn+1}^n: \mathcal{C}([0,1])\longrightarrow M_{q_n^{r_n}}(\mathcal{J}_{q_n}^{1}) \\
\hspace{2,22cm} f\longmapsto f(0)\otimes 1_{M_{q_n^{r_n+1}}}
\end{array}
\]
$\bullet$ \emph{The $(n+2)$-th partial $^*$-homomorphism for $\phi_{nn+1}$ and $\psi_{nn+1}$}. We define
\[
\begin{array}{ll}
\hspace{-0,9cm}\phi_{nn+1}^{n+1},\psi_{nn+1}^{n+1}: \mathcal{C}([0,1])\longrightarrow \mathcal{C}([0,1])\\
\hspace{2,42cm}f\longmapsto f(d_n)
\end{array}
\]
We can now construct
\[
\left\{
    \begin{array}{ll}
    		\phi_{nn+1}:=(\phi_{nn+1}^0,..., \phi_{nn+1}^{n-1},(\phi_{nn+1}^{n}, \phi_{nn+1}^{n+1})).\\
    		\psi_{nn+1}:=(\psi_{nn+1}^0,..., \psi_{nn+1}^{n-1},(\psi_{nn+1}^{n}, \psi_{nn+1}^{n+1})).
    \end{array}
\right.
\]
Finally, we define
\[
\left\{
    \begin{array}{ll}
    		A:=\lim\limits_{\underset{n}{\longrightarrow}}(A_n,\phi_{nn+1}).\\
    		B:=\lim\limits_{\underset{n}{\longrightarrow}}(B_n,\psi_{nn+1}).
    \end{array}
\right.
\]

\begin{prop}
Both $A$ and $B$ are separable unital $\CatCa$-algebras of with stable rank one.
\end{prop}

\begin{proof}
Since all $\CatCa$-algebras of the inductive systems are separable and unital, together with the fact that all morphisms are also unital, we easily obtain that $A$ and $B$ are unital separable $\CatCa$-algebras. In addition, the stable rank one property is preserved by inductive limits and any one-dimensional $\NCCW$ complex has stable rank one. 
\end{proof}


\subsection{Computation of invariants}\label{sec:invariant}
Now that we have built our inductive systems, let us show that they do not agree on their unitary Cuntz semigroup even though they have isomorphic $\K$-theory and isomorphic Cuntz semigroups. For that matter, we study the lattice of (closed two-sided) ideals of both algebras $A$ and $B$. We take a particular interest in the subset of simple ideals of $A$ and $B$. The benefit is twofold: the structure of simple ideals play a key role in proving that $\Cu_1(A)\nsimeq \Cu_1(B)$. Moreover, the direct sum of all simple ideals is a primitive ideal whose quotient is $\mathbb{C}$ and therefore, we can obtain the $\K$-theory of $A$ and $B$ using the 6-term exact sequence of $\K$-theory.
\begin{thm} 
\label{lma:simpleideals}
Let $i\in\N$. We consider
\[
\left\{
    \begin{array}{ll}
    		\mathfrak{s}_{i}:=\lim\limits_{\underset{n> i}\longrightarrow}(I_{n,i},{\phi_{nm}}_{|I_{n,i}}) \hspace{0,3cm}\text{ and } \hspace{0,3cm}\mathfrak{a}_{i}:=\lim\limits_{\underset{m> n}\longrightarrow}(I^c_{n,i},{\phi_{nm}}_{|I^c_{n,i}}).\\
    		\mathfrak{t}_{i}:=\lim\limits_{\underset{n> i}\longrightarrow}(J_{n,i},{\psi_{nm}}_{|J_{n,i}}) \hspace{0,3cm}\text{ and } \hspace{0,3cm}\!\mathfrak{b}_{i}:=\lim\limits_{\underset{n> i}\longrightarrow}(J^c_{n,i},{\psi_{nm}}_{|J^c_{n,i}}).
    \end{array}
\right.
\]
where $I_{n,i}:=M_{[n,i]}(\mathcal{I}_{q_i}^{n-i})$ is the $(i+1)$-th full block of $A_n$ and $I^c_{n,i}:=\overset{n-1}{\underset{j=0,j\neq i}\bigoplus}M_{[n,j]}(\mathcal{I}_{q_j}^{n-j})\,\oplus\, M_{[n,n]}(\mathcal{C}([0,1]))$ its complementary in $A_n$. Respectively, $J_{n,i}$ is the $(i+1)$-th full bock of $B_n$ and $J^c_{n,i}$ is its complementary in $B_n$. Then:

(i) The sets of simple ideals of $A$ and $B$ are respectively $\{\mathfrak{s}_{i}\}_{i\in\N}$ and $\{\mathfrak{t}_{i}\}_{i\in\N}$. 

(ii) We have that
\[
    \left\{
    \begin{array}{ll}
    		A=\mathfrak{a}_i\oplus \mathfrak{s}_{i}.\\
    		B=\mathfrak{b}_i\oplus \mathfrak{t}_{i}.  
    		\end{array}
\right.
\vspace{0,2cm}\]

(iii) The quotients $A/\!\underset{i\in\N}\oplus\mathfrak{s}_{i}\simeq B/\!\underset{i\in\N}\oplus\mathfrak{t}_{i}\simeq \mathbb{C}$.
\end{thm}

\begin{proof}
We prove all statements for $A$, and the statements for $B$ are proven similarly.

(i) We first describe a sufficient condition for an ideal of inductive system to be simple. (This is inspired by \cite[Lemma 1]{G92}.) Let $A:=\lim\limits_{\longrightarrow}(A_n,\phi_{nm})_n$ be an inductive limit in $\CatCa$ and let $I$ be a (closed two-sided) ideal of $A$. Using \cite[Lemma 4.5]{Bl77}, we know that $I_0:=\underset{n\in\N}\cup(\phi_{n\infty}(A_n)\cap I)$ is dense in $I$. Hence it is enough to check that the algebraic limit $I_0$ is simple. For any $n\in\N$, write $I_n:= \phi_{n\infty}^{-1}(\{\phi_{n\infty}(A_n)\cap I\})$. Observe that $I_n$ is an ideal of $A_n$ (hence a $\CatCa$-algebra) and that $I_0=\underset{n\in\N}\cup\phi_{n\infty}(I_n)$. Thus, to show that $I$ is a simple ideal of $A$, it is enough to show that for any $n\in\N$, any $x\in I_n$, then there exists $m\geq n$ such that $\overline{I_m\phi_{nm}(x)I_m}=I_m$.

 Now, let us prove that $\mathfrak{s}_{i}$ is simple. It is sufficient to show that for any $g\in I_{n,i}$ such that $g\neq 0$, then there exists $m>n$ such that $\overline{I_{m,i} \phi_{nm}(g)I_{m,i}}= I_{m,i}$.
Let $n>i$ and let $g\in I_{n,i}$ such that $g\neq 0$. Observe that $\{\{1/q_i^{r_m},\dots, (q_i^{r_m}-1)/q_i^{r_m}\}\}_{m>n}$ is dense in $[0,1]$. Also note that if the support of $g$ contains an interval of length $s$, then so does the support of $\phi_{nm}(g)$. We deduce that there exists $m\geq n$ such that $k/q_i^{r_m}\in \supp(\phi_{nm}(g))$ for some $1\leq k \leq (q_i^{r_m}-1)$. Therefore, the ideal generated by $\phi_{nm}(g)$ is dense in $I_{m,i}$, from which we deduce that $\mathfrak{s}_{i}$ is a simple $\CatCa$-algebra for any $i\in\N$.

Finally, by \cite[Lemma 4.5]{Bl77}, we know that any ideal of $A$ is in fact of the form $\lim\limits_{\underset{n}{\longrightarrow}}(I,{\phi_{nm}}_{|I})$ for some $I\in\Lat(A_n)$. It follows from the above that $\{\mathfrak{s}_i\}_i$ are the only simple ideals of $A$.

(ii) Notice that $\mathfrak{s}_i\cap\mathfrak{a}_i=\{0\}$ and that $\mathfrak{s}_i +\mathfrak{a}_i=A$ for any $i\in\N$.

(iii) Consider $(\mathbb{C},e_n)_{n\in\N}$, where $e_n:A_n\longrightarrow \mathbb{C}$ given by $e_n(f_0,...,f_n)=f_n(d_n)$. It is clear that $(\mathbb{C},e_n)$ is a cocone to the inductive system. We deduce that there exists a unique $^*$-homomorphism $e:A\longrightarrow \mathbb{C}$ satisfying the universal properties of the direct limit. It is trivial that $e$ is surjective and also that $\underset{i\in\N}\oplus\mathfrak{s}_{i}\subseteq \ker e$. Finally, since $\ker e$ is a closed two-sided ideal of $A$ and also using observations made from \cite[Lemma 4.5]{Bl77}, we obtain that $\underset{i\in\N}\oplus\mathfrak{s}_{i}\supseteq \ker e$. (In fact, one could argue saying that $\underset{i\in\N}\oplus\mathfrak{s}_{i}$ is a maximal closed two-sided ideal of $A$.) We conclude that $A/\!\underset{i\in\N}\oplus\mathfrak{s}_{i}\simeq\mathbb{C}$.
\end{proof}

\begin{lma}
\label{lma:Z/pq} Let $p_1,p_2$ be distinct prime numbers. Consider the inductive system \break$(\Z/p_1p_2\Z, \varphi_{p_1})_{n\in\N}$ in the category of abelian group, where $\varphi_{p_1}:\Z/p_1p_2\Z\longrightarrow \Z/p_1p_2\Z$ is the (additive) group morphism corresponding to the multiplication by $p_1$. Let us denote its inductive limit by $S$. 
Then $S\simeq \Z/p_2\Z$.
\end{lma}

\begin{proof}
We know that $p_1$ and $p_2$ are coprime. Hence, by the well-known Chinese remainder theorem, we also know that $\Z/p_1p_2\Z\simeq \Z/p_1\Z\times\Z/p_2\Z$. Moreover, $\Z/p_2\Z \overset{\times p_1}\longrightarrow \Z/p_2\Z $ is an isomorphism -which corresponds to a permutation in $\Z/p_2\Z$-, and $\Z/p_1\Z \overset{\times p_1}\longrightarrow \Z/p_1\Z $ is the zero morphism. Therefore, we have that $\im\varphi_{p_1}\simeq \Z/p_1\Z$ and $\ker\varphi_{p_1}\simeq \Z/p_2\Z$. In the end, the inductive system considered is naturally isomorphic to $(\Z/p_1\Z\times\Z/p_2\Z\overset{0\times \varphi_{p_1}}\longrightarrow \Z/p_1\Z\times\Z/p_2\Z\overset{0\times \varphi_{p_1}}\longrightarrow...) $ from which the result follows.
\end{proof}

\begin{thm}\label{cor:K1iso}
The $\CatCa$-algebras $A$ and $B$ have isomorphic $\K$-Theory. That is,
\[
\begin{array}{ll}
(i)\hspace{0,4cm} \K_0(A)\simeq \K_0(B).\\
(ii)\hspace{0,32cm} \K_1(A)\simeq \K_1(B).
\end{array}
\]
\end{thm}

\begin{proof}
Let $n\in\N$ and let $0\leq i\leq n-1$. It is not hard to check  that $K_0(\phi_{nn+1}^i):\Z\overset{\times q_i^{r_n+1}}\longrightarrow\Z$ and that $K_1(\phi_{nn+1}^i):\Z/q_i\Z\overset{\times p_{i-1}}\longrightarrow\Z/q_i\Z $. Similarly, $K_0(\psi_{nn+1}^i):\Z\overset{\times q_i^{r_n+1}}\longrightarrow \Z$ and $K_1(\psi_{nn+1}^i):\Z/q_i\Z\overset{\times p_{i}}\longrightarrow\Z/q_i\Z$. 
Doing a quick computation for $\K_0$ and using \autoref{lma:Z/pq} for $\K_1$, we obtain that
\[
\begin{array}{ll}\K_0({\mathfrak{s}_i})\simeq \K_0({\mathfrak{t}_i})\simeq \Z[\frac{1}{q_i}]\\\hspace{-1,7cm}\K_1({\mathfrak{s}_i})\simeq \Z/p_i\Z \hspace{2cm}\K_1({\mathfrak{t}_i})\simeq \Z/p_{i-1}\Z
\end{array}
\]
(Note that we have fixed $p_{-1}:=1$, which matches the fact that $\K_1({\mathfrak{j}_0})\simeq \{0\}$.)

We deduce the $\K$-theory of the respective (primitive) ideals $\underset{i\in\N}\oplus \mathfrak{s}_i, \underset{i\in\N}\oplus \mathfrak{t}_i $ of $A$ and $B$. More specifically, we immediately have $\K_0(\underset{i\in\N}\oplus \mathfrak{s}_i)\simeq \K_0(\underset{i\in\N}\oplus \mathfrak{t}_i) $ and the shift that adds a $0$ at the beginning of a sequence gives us an isomorphism between $\K_1(\underset{i\in\N}\oplus \mathfrak{s}_i)\simeq \K_1(\underset{i\in\N}\oplus \mathfrak{t}_i)$. Further, we know that $A/\!\underset{i\in\N}\oplus \mathfrak{s}_i \simeq \mathbb{C}$, and hence the canonical short-exact sequence of the ideal $\underset{i\in\N}\oplus \mathfrak{s}_i$ is in fact $0\longrightarrow \underset{i\in\N}\oplus \mathfrak{s}_i \longrightarrow A \longrightarrow \mathbb{C} \longrightarrow 0$. By the 6-term exact sequence (see e.g. \cite[Theorem 9.3.1]{Bl86}), we have the following exact diagram:
\[
\xymatrix{
\K_1(\underset{i\in\N}\oplus \mathfrak{s}_i)\ar[r] & \K_1(A)\ar[r] & 0\ar[d]\\
\Z\ar[u]^{\delta_1} & \K_0(A)\ar[l]& \K_0(\underset{i\in\N}\oplus \mathfrak{s}_i)\ar[l]
}
\]
Since $\pi:A\longrightarrow \mathbb{C}$ is unital and $\K_0(\mathbb{C})$ is generated by $[1_\mathbb{C}]$, we deduce that $\K_0(A)\longrightarrow \Z$ is surjective and we get that $\ker(\delta_{1})\simeq \Z$. This gives us $\delta_1=0$ and after doing similar computation for $B$, we deduce that
\[
\begin{array}{ll}
(i)\hspace{0,4cm} 0\longrightarrow	\K_0(\underset{i\in\N}\oplus \mathfrak{s}_i)\longrightarrow \K_0(A)\longrightarrow  \Z\longrightarrow 0 \text{ is exact.}
\\
(ii) \hspace{0,3cm}\K_1(A)\simeq \K_1(\underset{i\in\N}\oplus \mathfrak{s}_i).
\end{array}
\]
Finally, $\Z$ is free (as a $\Z$-module), so the exact sequence of (i) is splitting and hence we conclude that
\[
\begin{array}{ll}
(i)\hspace{0,4cm} \K_1(A)\simeq \K_1(\underset{i\in\N}\oplus \mathfrak{s}_i) \simeq \K_1(\underset{i\in\N}\oplus \mathfrak{t}_i)\simeq \K_1(B).\\
(ii) \hspace{0,3cm} \K_0(A)\simeq \K_0(\underset{i\in\N}\oplus \mathfrak{s}_i)\oplus \Z \simeq \K_0(\underset{i\in\N}\oplus \mathfrak{t}_i)\oplus \Z\simeq \K_0(B).
\end{array}
\]
\end{proof}

The next invariant to compute is the Cuntz semigroup. However, such a task is usually very tricky if not impossible, especially when it comes to inductive limits. Therefore, we shall use an analogous version of the approximate intertwining theorem for the category $\Cu$, obtained by the author in \cite{C22}. We mention that the latter needs quite long preliminary notations and hypotheses. Nevertheless, we observe that $\Cu(A_n)\simeq \Cu(B_n)$ for any $n\in\N$, and this fact simplifies the conditions to satisfy in order to get an approximate intertwining between the sequences. Therefore, we shall adapt the theorem, that we recall next, to our specific case and we refer the reader to the original paper for the general case.

As for notations, we denote the uniform basis of size $2^\infty$ of all the semigroups $\Cu(A_n),\Cu(B_n)$ by $(M_j,\epsilon_j)_j$ (see \autoref{sec:EK}) and we do not distinguish to which semigroup we refer to. (We will explicitly precise the latter if needed.) Moreover, for any $n\in\N$ and any $0\leq i\leq n+1$, we write $\alpha^i_{nn+1}:=\Cu(\phi_{nn+1}^i)$ and $\beta^i_{nn+1}:=\Cu(\psi_{nn+1}^i)$. We also write $\alpha_{nm}:=\Cu(\phi_{nm})$ and $\beta_{nm}:=\Cu(\psi_{nm})$ for any $n\leq m$. Using matrix representation, we can rewrite
\vspace{0,2cm}\[
\alpha_{nn+1}:=\left(\begin{smallmatrix}\alpha^0_{nn+1} & & & & \\  &&\ddots &&\\  &&&\ddots&\\ & & & & \alpha^n_{nn+1} \\  & &  & & \alpha^{n+1}_{nn+1} \end{smallmatrix}\right)
\hspace{1cm} \beta_{nn+1}:=\left(\begin{smallmatrix}\beta^0_{nn+1} & & & & \\  &&\ddots &&\\  &&& \ddots&\\ & & & & \beta^n_{nn+1} \\  & &  & & \beta^{n+1}_{nn+1} \end{smallmatrix}\right)
\vspace{0,2cm}\]
Finally, we recall that if $S:=\underset{r}{\oplus}\, S_i$ where each $S_i$ has a uniform basis $\mathcal{B}_i$, then the concatenation of the bases $\mathcal{B}_i$ is a uniform basis of $S$. (See \cite[Proposition 3.3]{C22}.) A direct consequence is that the associated semi-metric to the uniform basis $\mathcal{B}$ is $dd_{\Cu,\mathcal{B}}:=\max\limits_r (dd_{\Cu,\mathcal{B}_i})$. 

In our case, we obtain that
\begin{align*}
dd_{\Cu,\mathcal{B}_{2^\infty}}(\alpha_{nn+1}, \beta_{nn+1})&=\max\limits_{0\leq i\leq n+1}(dd_{\Cu,\mathcal{B}_{2^\infty}}(\alpha^i_{nn+1}, \beta^i_{nn+1}))\\
&=\max\limits_{0\leq i\leq n-1}(dd_{\Cu,\mathcal{B}_{2^\infty}}(\alpha^i_{nn+1}, \beta^i_{nn+1})).
\end{align*}
Again, we refer the reader to \autoref{sec:EK} for the construction of $dd_{\Cu,\mathcal{B}_{2^\infty}}$.

\begin{thm}\cite[Theorem 3.15]{C22}
\label{thm:approxinter}
Let $(\Cu(A_n),\alpha_{nm})_{n\in\N}$ and $(\Cu(B_n),\beta_{nm})_{n\in\N}$ be the inductive sequences obtained by applying $\Cu$ to the inductive sequences of $A$ and $B$ respectively. Assume that there exists a strictly increasing sequence $(j_n)_n$ of natural numbers such that:

(i) For any $n\leq m$, we have $\alpha_{nm}(M_{j_n})\subseteq M_{j_m}$ and $\beta_{nm}(M_{j_n})\subseteq M_{j_m}$.

(ii) For any $n\in\N$, we have that $dd_{\Cu,\mathcal{B}_{2^\infty}}(\alpha_{nn+1}, \beta_{nn+1})< 1/2^{j_n}$.

Then $\Cu(A)\simeq \Cu(B)$.
\end{thm}

\begin{lma}
For any $n\in\N$, we have $dd_{\Cu,\mathcal{B}_{2^\infty}}(\alpha_{nn+1}, \beta_{nn+1})\leq 1/2^n $.
\end{lma}

\begin{proof}
Let $n\in\N$. We already know that $\alpha^{n}_{nn+1} = \beta^{n}_{nn+1}$ and that $\alpha^{n+1}_{nn+1} = \beta^{n+1}_{nn+1}$ since the $^*$-homomorphisms are identical. This leads to $dd_{\Cu,\mathcal{B}_{2^\infty}}((\alpha^n_{nn+1},\alpha^{n+1}_{nn+1}),(\beta^n_{nn+1},\beta^{n+1}_{nn+1}))=0$. We are going to show that $dd_{\Cu,\mathcal{B}_{2^\infty}}(\alpha^i_{nn+1}, \beta^i_{nn+1})\leq 1/q_i^{r_n-2}$ for any $0\leq i\leq n$ and the result will follow.

Let $0\leq i\leq n-1$. We define $l_{n,i}:=\max\limits_{l\in\N} \{l\mid 2^l\leq q_i^{r_n}/q_i\}$. Observe that $q_i/q_i^{r_n}\leq 1/2^{l_{n,i}}\leq 1/q_i^{r_n-2}$. Now let $\{\overline{U_k}\}_{1}^{2^{l_{n,i}}}$ be the finite closed cover of $[0,1]$ of size $2^{l_{n,i}}$ as constructed in \autoref{sec:uniformbasis}. Let us write
\[
\begin{array}{ll}
c_0:= \card \,(\,\,\{1/q_i^{r_n},\dots,(q_i^{r_n}-1)/q_i^{r_n}\}\cap ([0,1]\setminus (\underset{k=1}{\overset{2^n}\cup}U_k))\,)\\
c_k:= \card\,(\,\{1/q_i^{r_n},\dots,(q_i^{r_n}-1)/q_i^{r_n}\}\cap U_k)
\end{array}
\]
for any $1\leq k \leq 2^{l_{n,i}}$.
From the way we chose $l_{n,i}$, we know that each open interval $U_k$ contains at least $q_i$ points of the set $\{1/q_i^{r_n},\dots,(q_i^{r_n}-1)/q_i^{r_n}\}$.  In other words, $c_k\geq q_i$ for any $k$.  Let us show $dd_{\Cu,\mathcal{B}_{2^\infty}}(\alpha^i_{nn+1}, \beta^i_{nn+1})\leq 1/2^{l_{n,i}}$. We start by observing that for any $h:=[f_h]\in \Lambda_{l_{n,i}}\subseteq\Cu(\mathcal{I}_{q_i}^{n-i})$ (see the last paragraph of \autoref{sec:EK}), we have
\[
\small\left\{
\begin{array}{ll}
\alpha^i_{nn+1}(h)= [f_h(\xi_0)\otimes e_A^i+ f_h(\xi_1)\otimes (1-e_A^i)]_{\Cu}+ \sum\limits_{k=1}^{2^{l_{n,i}}}c_k (h_{U_k}{\frac{q_i}{q_i^{n-i+1}}1_{[0,1]}})+ \sum\limits_{x\in c_0}(h_x{\frac{q_i}{q_i^{n-i+1}}1_{[0,1]}})\\
\beta^i_{nn+1}(h)= [f_h(\xi_0)\otimes e_B^i+ f_h(\xi_1)\otimes (1-e_B^i)]_{\Cu}+ \sum\limits_{k=1}^{2^{l_{n,i}}}c_k (h_{U_k}{\frac{q_i}{q_i^{n-i+1}}1_{[0,1]}})+ \sum\limits_{x\in c_0}(h_x{\frac{q_i}{q_i^{n-i+1}}1_{[0,1]}})
\end{array}
\right.
\]
where $h_{U_k},h_x\in\{0,\dots, q_i\}$ are respectively defined by the values of $h$ on $U_k$ and at $x$, multiplied by $q_i^{n-i}$ . In other words, we have that $h_{|U_k}=h_{U_k}/q_i^{n-i}$ and $h(x)=h_x/q_i^{n-i}$.

Now, let $h', h$ be elements of $\Lambda_{l_{n,i}}$ such that $h'\ll h$.  In particular, $0\leq (h'\ll)\, h$ and that $h' (\ll h) \leq q_i.(\frac{1}{q_i^{n-i}}.1_{[0,1]})$. Combined with the fact that $[1_{\mathcal{I}_{q_i}^{n-i}}(\xi_0)\otimes e_A^i+ 1_{\mathcal{I}_{q_i}^{n-i}}(\xi_1)\otimes (1-e_A^i)]_{\Cu}=[1_{\mathcal{I}_{q_i}^{n-i}}(\xi_0)\otimes e_B^i+ 1_{\mathcal{I}_{q_i}^{n-i}}(\xi_1)\otimes (1-e_B^i)]_{\Cu}=q_i(\frac{q_i}{q_i^{n-i+1}}.1_{[0,1]})$ and

 $[0_{\mathcal{I}_{q_i}^{n-i}}(\xi_0)\otimes e_A^i+ 0_{\mathcal{I}_{q_i}^{n-i}}(\xi_1)\otimes (1-e_A^i)]_{\Cu}=[0_{\mathcal{I}_{q_i}^{n-i}}(\xi_0)\otimes e_B^i+ 0_{\mathcal{I}_{q_i}^{n-i}}(\xi_1)\otimes (1-e_B^i)]_{\Cu}=0$, we obtain the following bounds:
\[
\begin{array}{ll}
\alpha^i_{nn+1}(h'), \beta ^i_{nn+1}(h')\leq q_i(\frac{q_i}{q_i^{n-i+1}}.1_{[0,1]})+ \sum\limits_{k=1}^{2^{l_{n,i}}}c_k (h'_{U_k}{\frac{q_i}{q_i^{n-i+1}}1_{[0,1]}})+ \sum\limits_{x\in c_0}(h'_x{\frac{q_i}{q_i^{n-i+1}}1_{[0,1]}}).\\
\alpha^i_{nn+1}(h), \beta^i_{nn+1}(h)\geq 0+ \sum\limits_{k=1}^{2^{l_{n,i}}} c_k  (h_{U_k}{\frac{q_i}{q_i^{n-i+1}}1_{[0,1]}})+\sum\limits_{x\in c_0} (h_x{\frac{q_i}{q_i^{n-i+1}}1_{[0,1]}}).
\end{array}
\]
It is sufficient to prove that the lower bound of $\alpha^i_{nn+1}(h), \beta^i_{nn+1}(h)$ is greater than the upper-bound of $\alpha^i_{nn+1}(h'), \beta ^i_{nn+1}(h')$. First, we immediately see that 
\[
\sum\limits_{x\in c_0} (h'_x{\frac{q_i}{q_i^{n-i+1}}1_{[0,1]}})\leq \sum\limits_{x\in c_0} (h_x{\frac{q_i}{q_i^{n-i+1}}1_{[0,1]}}).
\]
Then, if $\supp h'=\supp h$ (or equivalently, if $h=h'=q_i.(\frac{1}{q_i^{n-i}} 1_{[0,1]})$), we easily see that $\alpha(h)=\beta(h)$. Else, we have that $\supp h'\subsetneq \supp h$ and thus we can find at least one interval $U_k$ such that $h'_{U_k}=0$ and $h_{U_k}\geq 1$. Combined with the fact that $c_k\geq q_i$, we deduce that
\[
 q_i(\frac{q_i}{q_i^{n-i+1}}1_{[0,1]})+ \sum\limits_{k=1}^{2^{l_{n,i}}} c_k (h'_{U_k}{\frac{q_i}{q_i^{n-i+1}}1_{[0,1]}})\leq \sum\limits_{k=1}^{2^{l_{n,i}}} c_k (h_{U_k}{\frac{q_i}{q_i^{n-i+1}}1_{[0,1]}}).
\]
 Putting everything together, we have  that for any $h', h\in\Lambda_{l_{n,i}}$ such that $h'\ll h$, then
\[\alpha^i_{nn+1}(h'), \beta ^i_{nn+1}(h')\leq \alpha^i_{nn+1}(h), \beta^i_{nn+1}(h).\]
In other words, $dd_{\Cu,\mathcal{B}_{2^\infty}}(\alpha^i_{nn+1},\beta^i_{nn+1})\leq 1/2^{l_{n,i}}\leq 1/q_i^{r_n-2}$. Finally, we know that $(r_k)_k$ is a strictly increasing sequence of natural numbers with $r_0\geq 2$. Thus, we conclude that 
\[
dd_{\Cu,\mathcal{B}_{2^\infty}}(\alpha_{nn+1}, \beta_{nn+1})=\max\limits_{i=0,..,n-1}(dd_{\Cu,\mathcal{B}_{2^\infty}}(\alpha^i_{nn+1},\beta^i_{nn+1}))\leq 1/q_0^{r_n-2}\leq 1/2^n.
\] 
\end{proof}

\begin{cor}
The approximate intertwining theorem gives us $\Cu(A)\simeq \Cu(B)$.
\end{cor}

\begin{proof}
We claim that the sequence $(j_n)_n$ defined by $j_n:=n$ satisfies the conditions (i) and (ii) of \autoref{thm:approxinter}. We only have to check (ii), since (i) is proven by the previous lemma. 

Let $n\in\N$, let $0\leq i\leq n-1$ and let $j\in\N$. One can check that we have $\alpha^i_{nn+1}(M_j) \subseteq M_{j-1}$ in $\Cu(A_{n+1})$ and $\beta^i_{nn+1}(M_j)\subseteq M_{j-1}$ in $\Cu(B_{n+1})$.
 Also, for any $g\in\Lsc([0,1],\N)$, the elements $\alpha^n_{nn+1}(g),\alpha^{n+1}_{nn+1}(g)$ and $\beta^n_{nn+1}(g),\beta^{n+1}_{nn+1}(g)$ are compact elements, and hence belong to all $M_j$ of their respective $\Cu$-semigroup. 
 
We conclude that $\alpha_{nm}(M_j)\subseteq M_j$ in $\Cu(A_m)$ and that $\beta_{nm}(M_j)\subseteq M_j$ in $\Cu(B_m)$, for any $n,m$ in $\N$ and any $j\in\N$, which ends the proof.
\end{proof}

We have seen that $A$ and $B$ agree on their $\K$-theory and their Cuntz semigroups. In the simple case, this would have been enough to conclude that $A$ and $B$ were isomorphic. Yet, we see in the sequel that $A\nsimeq B$ since they are distinguish by the unitary Cuntz semigroup. The conclusion is twofold: the data from the original Elliott invariant and/or the Cuntz semigroup is not enough to classify non-simple $\CatCa$-algebras (that have $\K_1$-obstruction) and the unitary Cuntz semigroup seems a promising candidate that brings extra data needed to obtain a classification result. Therefore, one could hope for a classification of a certain class of inductive limits of one-dimensional $\NCCW$-complexes by means of their unitary Cuntz semigroup. 

In order to show that $\Cu_1(A)\nsimeq\Cu_1(B)$, an explicit computation of $\Cu_1(A)$ and $\Cu_1(B)$, far more complex that the one of $\Cu(A)$ and $\Cu(B)$, would be nearly impossible. And even though it is not beyond reason to expect that an approximate intertwining theorem for unitary Cuntz semigroups can be develop at some point, we have to take a different approach for now. Instead, we are using the structure of the simple ideals of $A$ and $B$ together with an exact diagram (in $\Cu^\sim$) linking the Cuntz semigroup, its unitary version and the $\K_1$-group that we recall now.

\begin{thm}
\label{prg:transposetoalgebras}\cite[Theorem 4.16]{C21b}
Let $A,B\in \CatCa$. Let $\phi:A\longrightarrow B$ be a $^*$-homomorphism. Let $I\in\Lat(A)$. Write $J:=\overline{B\phi(I)B}$, the smallest ideal of $B$ containing $\phi(I)$ and $\alpha:=\Cu_1(\phi)$. We also write $\alpha_{|\Cu_1(I)}:\Cu_1(I)\longrightarrow \Cu_1(J)$, ${\alpha_0}_{|\Cu(I)}:\Cu(I)\longrightarrow \Cu(J)$ and $\alpha_I: \K_1(I)\longrightarrow \K_1(J)$. Then the following diagram is commutative with exact rows:
\[
\xymatrix{
0\ar[r]^{} & \Cu(I)\ar[d]_{{\alpha_0}_{|\Cu(I)}}\ar[r]^{} & \Cu_1(I) \ar[d] _{\alpha_{|\Cu_1(I)}}\ar[r]^{} & \K_1(I) \ar[d]^{\alpha_I}\ar[r]^{} & 0
\\
0\ar[r]^{} & \Cu(J)\ar[r]_{} & \Cu_1(J)\ar[r]_{} & \K_1(J)\ar[r]^{} & 0
} 
\]
Furthermore, if $\alpha$ is a $\Cu^\sim$-isomorphism, then $\alpha(\Cu_1(I))=\Cu_1(J)$ and $\alpha_{|\Cu_1(I)}
$ is a $\Cu^\sim$-isomorphism. A fortiori, we also have ${\alpha_0}_{|\Cu(I)}
$ is a $\Cu$-isomorphism and $\alpha_I
$ is a $\AbGp$-isomorphism.
\end{thm}
\begin{thm}
\label{thm:pepite}
There is no $\Cu^\sim$-isomorphism between $\Cu_1(A)$ and $\Cu_1(B)$. A fortiori, $A\nsimeq B$.
\end{thm}

\begin{proof}
Suppose there exists a $\Cu^\sim$ isomorphism $\gamma:\Cu_1(A)\longrightarrow \Cu_1(B)$. Then, there is a lattice isomorphism between $\Lat(\Cu_1(A))\simeq\Lat(\Cu_1(B))$ given by $\Cu_1(I)\longmapsto \gamma(\Cu_1(I))$. That is, for any (simple) $I\in\Lat(A)$, there exists a unique (simple) $J_I\in\Lat(B)$ such that $\gamma_{|\Cu_1(I)}:\Cu_1(I)\longrightarrow \Cu_1(J_I)$ is a $\Cu^\sim$-isomorphism. In addition, we know that the following diagram is row-exact and commutative:
\[
\xymatrix{
0\ar[r]^{} & \Cu(I)\ar[d]_{\nu_{+}(\gamma_{|\Cu_1(I)})}\ar[r]^{} & \Cu_1(I) \ar[d]_{\gamma_{|\Cu_1(I)}}\ar[r]^{} & \K_1(I) \ar[d]^{\nu_{max}(\gamma_{|\Cu_1(I)})}\ar[r]^{} & 0
\\
0\ar[r]^{} & \Cu(J_I)\ar[r]_{} & \Cu_1(J_I)\ar[r]_{} & \K_1(J_I)\ar[r]^{} & 0
} 
\] 
On the other hand, we have proved in \autoref{lma:simpleideals} that the simple ideals of $\Cu_1(A)$ and $\Cu_1(B)$) are respectively $\{\Cu_1(\mathfrak{s}_i)\}_{i\in\N}$ and $\{\Cu_1(\mathfrak{t}_i)\}_{i\in\N}$. Therefore, we deduce that for any $i\in\N$ then there exists a unique $j\in\N$ such that $\gamma_{|\mathfrak{s}_i}:\Cu_1(\mathfrak{s}_i)\longrightarrow \Cu_1(\mathfrak{t}_j)$ is a $\Cu^\sim$-isomorphism. Furthermore, by the diagram above, we know that $\gamma_{|\mathfrak{s}_i}$ induces the two following isomorphisms: 
\[
(*)\left\{
    \begin{array}{ll}
    		(\gamma_{|\mathfrak{s}_i})_+:\Cu(\mathfrak{s}_i)\simeq \Cu(\mathfrak{t}_j) \text{ in } \Cu.\\
    		(\gamma_{|\mathfrak{s}_i})_{max}:\K_1(\mathfrak{s}_i)\simeq \K_1(\mathfrak{t}_j) \text{ in AbGp}. 
    \end{array}
\right.
\]
Nevertheless, in the proof of \autoref{cor:K1iso}, we have computed that
\[
\begin{array}{ll}\K_0({\mathfrak{s}_i})\simeq \K_0({\mathfrak{t}_i})\simeq \Z[\frac{1}{q_i}]\\\hspace{-1,7cm}\K_1({\mathfrak{s}_i})\simeq \Z/p_i\Z \hspace{2cm}\K_1({\mathfrak{t}_i})\simeq \Z/p_{i-1}\Z
\end{array}
\]
from which we obtain that
\[
\left\{
    \begin{array}{ll}
    		\K_0(\mathfrak{s}_i)\simeq \K_0(\mathfrak{t}_j) \text{ if and only if } i=j.\\
    		\K_1(\mathfrak{s}_i)\simeq \K_1(\mathfrak{t}_j) \text{ if and only if } i+1=j. 
    \end{array}
\right.
\]
We hence arrive to a contradiction since $j$ has to be equal to both $i$ and $i+1$ in order to satisfy the necessary conditions of $(*)$. We conclude that $\Cu_1(A)\nsimeq\Cu_1(B)$ and hence $A\nsimeq B$.
\end{proof}

\section{Further Remarks} 
\emph{Remark 1.} The alternative picture of the unitary Cuntz semigroup of a $\CatCa$-algebra of stable rank one described in \autoref{prg:newpicture} can be used for any separable $\CatCa$-algebra $A$. The proof given in \cite[\S 4.1]{C21} relies upon the fact that whenever $A$ is separable then $\K_1(i):\K_1(\her(a))\simeq \K_1(I_a)$ for any $a\in (A\otimes\mathcal{K})_+$, where $i:\her(a)\lhook\joinrel\longrightarrow I_a$ is the canonical injection from the hereditary $\CatCa$-algebra generated by $a$ into the ideal generated by $a$. (See \cite[Theorem 2.8]{Br77}.) Nevertheless, the author has been informed by J. Gabe that this is true in general, without having to assume separability. Unfortunately, the proof has not been published but it remains true that the alternative picture of the unitary Cuntz semigroup is valid for any $\CatCa$-algebra of stable rank one.

\emph{Remark 2.} It would be worth knowing whether other existing invariants for non-simple $\CatCa$-algebras would be able to distinguish $A$ and $B$. Among them, we select the followings:

$\bullet$ \emph{The invariant $\K_*$} introduced by Gong, Jiang and Li in \cite{GJL19} and \cite{GJL20}. This invariant based on the original Elliott invariant together with total $\K$-Theory, affine functions from the tracial space of corner algebras and their Hausdorffified $\K_1$-groups, is a complete invariant for $\AH$-algebras with the ideal property and no dimension growth. Thus, it is sufficient to check that $A$ and $B$ belong to the the latter classifiable class of $\CatCa$-algebras. Clearly, both inductive systems have no dimension growth. It would be left to prove that $A$ and $B$ have the ideal property. Nevertheless, the proof of \autoref{lma:simpleideals} gives the impression that this property is verified and hence, there is a fair hope that we could conclude that the invariant $\Inv$ also distinguishes these algebras. Let us mention that $\Inv$ and $\Cu_1$ are quite afar from one another and so are techniques/arguments involved. As a result, it would be more challenging to directly prove that $\Inv(A)\nsimeq \Inv(B)$ due to the complex nature of the invariant $\Inv$.

$\bullet$ \emph{The invariant $\K_{\text{\tiny\varhexagon}}$} introduced by M. R$\o$rdam in \cite{R97}, see also \cite{RR07}. This invariant is based on the $\K$-theory of ideals through the 6-terms exact sequence. It seems that using a similar approach as in the proof of \autoref{thm:pepite}, we would obtain a contradiction when trying to find an isomorphism between hexagonal exact sequences obtained from the canonical extensions of the simple ideals $\{\mathfrak{s}_i\}_i$ and $\{\mathfrak{t}_i\}_i$, to conclude that $\K_{\text{\tiny\varhexagon}}(A)\nsimeq \K_{\text{\tiny\varhexagon}}(B)$. In contrast to $\Inv$, the arguments involved here are similar and at the time of writing, it remains unknown whether the $\Cu_1$-semigroup is strictly stronger than $\K_{\text{\tiny\varhexagon}}$, in the sense that there exist two $\CatCa$-algebras $A$ and $B$ such that $\K_{\text{\tiny\varhexagon}}(A)\simeq \K_{\text{\tiny\varhexagon}}(B)$ and such that $\Cu_1(A)\nsimeq\Cu_1(B)$. It is likely that inspiration can be drawn from the examples constructed in this article to obtain such $\CatCa$-algebras. As a matter of fact, these constructions could give rise to a valuable source of examples and counterexamples in the future, not only to prove that $\Cu_1$ and $\K_{\text{\tiny\varhexagon}}$ are not isomorphic but also to study more in-depth and reveal the information encoded in the unitary Cuntz semigroup.

$\bullet$ \emph{The invariant $\Cu_{\T}$} introduced by Antoine, Dadarlat, Perera and Santiago in \cite{ADPS14}. This invariant is obtained by taking the Cuntz semigroup of the tensor product of $\mathcal{C}(\T)$ with a $\CatCa$-algebra $A$. Under  the condition of simplicity (and other assumptions), this invariant contains the information of the $\K_1$-group of the latter $\CatCa$-algebra, which makes this invariant a candidate for distinguishing $A$ and $B$. Due to the complexity of computation of the Cuntz semigroup, it would be helpful to use the approximate intertwining theorem described in \cite[Theorem 3.15]{C22} to conclude. Depending on whether $\Cu_\T(A)$ is isomorphic to $\Cu_\T(B)$ or not, one should also study the link between these two invariants. 

\emph{Remark 3.} The author is investigating the definition of a semi-metric on \break$\Hom_{\Cu^\sim}(\Cu_1(A),\Cu_1(B))$. Later, one could generalize the approximate intertwining theorem for (concrete) unitary Cuntz semigroups and ensure that $\Cu_1$ distinguishes $A$ and $B$ by another argument.

\emph{Remark 4.} Most recently, the author has been informed of an unpublished manuscript (see \cite{AL22}) where the authors develop a refined version of the Cuntz unitary semigroup that includes the total $\K$-theory. There is no doubt that this new invariant is taking benefits from both invariants $\Inv$ and $\Cu_1$, not to speculate that the two latter can be recovered from the former. Therefore, this \emph{total version} of the unitary Cuntz semigroup seems very promising and it would be worth seeing if this invariant is complete for a certain class of non-simple inductive limits of one-dimensional $\NCCW$ complexes. 

Additionally, it is worth mentioning that the manuscript \cite{AL22} points out a flaw appearing in the very last part of \cite{C21}. Namely, the present author proved in \cite[Theorem 5.20]{C21} that $\K_*$ could be fully recovered from $\Cu_1$. As a direct corollary of the classification results obtained in \cite[Theorems 7.1-7.3-7.4]{E93}, he stated that the unitary Cuntz semigroup classifies homomorphisms of unital $\A\!\T$-algebras of real rank zero and that the unitary Cuntz semigroup is a complete invariant for unital $\AH_d$ algebras of real rank zero. Nevertheless, \cite[Theorem 7.1]{E93} has been disproved some years later in \cite{EGS98}. (See Example 2.19.) We hence take the opportunity to write an erratum.

\section*{Erratum of \texorpdfstring{\cite[Section 5]{C21}}{[10,Section 5]}}

In this erratum, we aim to fix \cite[Corollary 5.21]{C21} and we refer the reader to the said article for definitions and notations. First, it has been observed in \cite[Example 2.2]{AL22} that the unitary Cuntz semigroup of a $\CatCa$-algebra of stable rank one does not have weak cancellation, in general. In other words, \cite[Proposition 5.16]{C21} is not correct. As a consequence, what is following the latter proposition, and in particular the proof of \cite[Theorem 5.20]{C21} (which stipulates that we can fully recover $\K_*$ from $\Cu_1$), is not quite right. To remedy this issue and prove that \cite[Theorem 5.20]{C21} is nonetheless true, we have to introduce a weaker version of weak cancellation for $\Cu^\sim$-semigroups that is satisfied by any $\CatCa$-algebra of stable rank one. Subsequently, we have to adapt the definition of the category $\Cu^\sim_u$, introduced in \cite[Definition 5.17]{C21}, together with the functor $H_*$, introduced in \cite[Lemma 5.19]{C21}. We require the objects of the category $\Cu^\sim_u$ to satisfy a couple of additional axioms (positive convexness, positive weak cancellation and assume that $\{0\}$ sits as an ideal inside $S$). We also slightly modify the positive cone of the ordered abelian group obtained through $H_*$. Thereafter, we are able to argue similarly as in the original proof of \cite[Theorem 5.20]{C21} to conclude that we can fully recover $\K_*$ from $\Cu_1$. 
We start by recalling a couple of definitions and properties that will be of use when redefining the category $\Cu^\sim_u$. We refer the reader to \cite{C21b} for more details.

\begin{dfn*}\cite[Definition 3.3]{C21b}\em \,Let $S$ be a $\Cu^\sim$-semigroup. We define the following axioms: 

(PD) We say that $S$ is \emph{positively directed} if, for any $x\in S$, there exists $p_x\in S$ such that $x+p_x\geq 0$.

(PC) We say that $S$ is \emph{positively convex} if, for any $x,y\in S$ such that $y\geq 0$ and $x\leq y$, we have $x+y\geq 0$.
\end{dfn*}

\begin{prop*}\cite[Proposition 3.12]{C21b} Let $S$ be positively directed and positively convex $\Cu^\sim$-semigroup. Let $x$ be a positive element of $S$. Then the ideal generated by $x$ exists and we have
\[
I_x=\{y\in S \mid \text{ there is } y'\in S \text{ with } 0\leq y+y'\leq\infty x\}
\vspace{-0cm}\]
\end{prop*}

It has been shown that $\Cu_1(A)$ satisfies axioms (PD) and (PC), for any $\CatCa$-algebra $A$ of stable rank one, and we easily see that that $I_{0_{\Cu_1(A)}}=\{0\}$. All these properties will be needed in the new proof of \cite[Theorem 5.20]{C21}. Nevertheless, they need not hold for abstract $\Cu^\sim$-semigroups. While it seems obvious for axioms (PD) and (PC), the last property is quite surprising. We here expose an example of a positively directed and positively convex $\Cu^\sim$-semigroup whose ideal generated by its neutral element is strictly larger than $\{0\}$.
Let $S:={\overline{\N}}\times\Z$ endowed with the component-wise sum and we define a partial order on $S$ as follows: for any two pairs $(g,k),(h,l)$ we have that $(g,k)\leq (h,l)$ if $g\leq h $ in ${\overline{\N}}$ and $k=l$ in $\Z$. Notice that $S_+= {\overline{\N}} \times \{0\}$. One can check that $(S,+,\leq)$ is a countably-based positively directed and positively convex $\Cu^\sim$-semigroup and by the above proposition, we compute that $I_0=\{0\}\times\Z$.

We proceed by introducing a weaker version of weak cancellation, respectively of cancellation of compact elements.

\begin{dfn*}
\em Let $S$ be a $\Cu^\sim$-semigroup. We define the following axioms:

(PWC) \hspace{0,3cm} We say that $S$ has \emph{positive weak cancellation} if $x+z\ll z$ implies that $x\ll 0$, for any $x,z\in S$.

(PCC) \hspace{0,4cm} We say that $S$ has \emph{positive cancellation of compact elements} if $x+z\leq z$ implies that $x\leq 0$, for any $x\in S$ and $z\in S_c$.
\end{dfn*}

\begin{rmk*}
\em It can be shown that these axioms are satisfied by the unitary Cuntz semigroup of any $\CatCa$-algebra of stable rank one and that axiom (PWC) implies axiom (PCC).
\end{rmk*}
We have now the right tools to adjust the category $\Cu^\sim_u$ and the functor $H_*$ that will allow us to recover $\K_*$ from $\Cu_1$. 

$\bullet\,$[The category $\Cu^\sim_u$] We say that a pair $(S,u)$ is a \emph{$\Cu^\sim$-semigroup with compact order-unit} if, $S$ is a positively directed, positively convex $\Cu^\sim$-semigroup that satisfies axiom (PWC) and such that $I_0=\{0\}$, and if $u$ is a compact order-unit for the positive cone $S_+$.  We say that a $\Cu^\sim$-morphism $\alpha:S\longrightarrow T$ between two $\Cu^\sim$-semigroups with compact order-unit $(S,u)$ and $(T,v)$ \emph{preserves the compact order-unit} if, $\alpha(u)\leq v$. The category of \emph{$\Cu^\sim$-semigroups with compact order-unit}, denoted $\Cu^\sim_u$, consists of $\Cu^\sim$-semigroups with compact order-unit together with $\Cu^\sim$-morphisms that preserve the compact order-unit. 

$\bullet\,$[The functor $H_*$] We consider the assignment 
\[
	\begin{array}{ll}
	\hspace{0cm} H_*:\Cu^\sim_{u}\longrightarrow \AbGp_{u}\\
	\hspace{0,6cm} (S,u)\longmapsto (\Gr(S_{c}),\iota(S_{c}),u)\\
		\hspace{1,22cm} \alpha \longmapsto \Gr(\alpha_{c})
	\end{array}
\] 
where $\iota:S_c\longrightarrow \Gr(S_c)$ is given by $x\longmapsto [(x,0)]_{\Gr(S_c)}$. 

First, we have to check that $(\Gr(S_{c}),\iota(S_{c}))$ is an ordered group. The fact that $\Gr(S_{c})=\iota(S_{c})-\iota(S_{c})$ is immediate from the Grothendieck construction. The non-trivial point is to prove that $\iota(S_{c})\cap(-\iota(S_{c}))=\{0\}$. Take $g\in \iota(S_{c})\cap(-\iota(S_{c}))$. There exist $x,y\in S_c$ such that $g=[(x,0)]=[(0,y)]$. It suffices to show that $x=y=0$. We start by observing that there exists $z\in S_c$ such that $x+y+z=z$ in $S_c$. Using axiom (PCC), we get that $x+y\leq 0$ which leads to $x+y=0$, by axiom (PC). Now using the fact that $I_0=\{x\mid \exists y, x+y=0\}=\{0\}$, we deduce that $x=y=0$. Therefore, we conclude that $(\Gr(S_{c}),\iota(S_{c}))$ is an ordered group. 
Next, arguing as in the proof of \cite[Lemma 5.19]{C21}, we know that $\alpha_c:S_c\longrightarrow T_c$ is a monoid morphism that induces a group morphism $\Gr(\alpha_c):\Gr(S_c)\longrightarrow \Gr(T_c)$. Moreover, for any element $[(x,0)]\in\iota(S_c)$, it is immediate that $\Gr(\alpha_c)([(x,0)])=[(\alpha_c(x),0)]$ belongs to the positive cone of $\Gr(T_c)$. In other words, $\Gr(\alpha_c)(\iota(S_c))\subseteq \iota(T_c)$. As a result $H_*:\Cu_u^\sim\longrightarrow \AbGp_u$ is a well-defined functor. 

$\bullet\,$[The natural isomorphism $\eta_*:H_*\circ\Cu_{1,u}\simeq \K_*$]
We first show that for any unital $\CatCa$-algebra of stable rank one $A$, there is a monoid isomorphism $\iota(\Cu_1(A)_{c})\simeq \K_*(A)_+$ and the result will follow. Since $A$ has stable rank one, we know that any element $x\in \Cu_1(A)_c$ is of the form $[(p,u)]$ for some projection $p\in A\otimes\mathcal{K}$ and some unitary $u\in \her(p)$. Hence we can construct the following exhaustive monoid morphisms:
\[
\begin{array}{ll}
\iota:\Cu_1(A)_c\relbar\joinrel\twoheadrightarrow\iota(\Cu_1(A)_{c})\hspace{3cm}\gamma:\Cu_1(A)_c\relbar\joinrel\twoheadrightarrow{\K_*(A)}_+\\
\hspace{0,7cm}[(p,u)]\longmapsto [(\,[(p,u)],[(0_A,1_\mathbb{C})]\,)]_{\Gr}\hspace{2cm}[(p,u)]\longmapsto ([p],[u+(1-p)]_{\K_1(A)})
\end{array}
\]
In order to reach the desired isomorphism, we first prove that 
\[
\iota([(p,u)])=\iota([(q,v)]) \text{ if and only if } \gamma([(p,u)])=\gamma([(q,v)]).
\] 
Let $[(p,u)]$ and $[(q,v)]$ be in $\Cu_1(A)_c$. We know that $\iota([(p,u)])=\iota([(q,v)])$ if and only if there exists $[(r,w)]\in \Cu_1(A)_c$ such that $[(p,u)]+[(r,w)]=[(q,v)]+[(r,w)]$ if and only if there exists $[(r,w)]\in \Cu_1(A)_c$ such that $[(p,u)]+[(r,w)]+[(1_A,1_A)]=[(q,v)]+[(r,w)]+[(1_A,1_A)]$. 

That is, $\iota([(p,u)])=\iota([(q,v)])$ if and only if there exists $[(r,w)]\in \Cu_1(A)_c$ such that $[r]\geq [1_A]$ and such that $[(p,u)]+[(r,w)]=[(q,v)]+[(r,w)]$. Which is in turn equivalent to 
\[
	(*)\left\{\begin{array}{ll}
[p]+[r]=[q]+[r] \text{ in} \Cu(A)\\
	\delta_{I_pA}([u]_{\K_1(I_p)})+[w]=\delta_{I_qA}([v]_{\K_1(I_q)})+[w] \text{ in } \K_1(A)
	\end{array}\right.
\] 
On the other hand, $\Cu(A)$ has cancellation of projections since $A$ has stable rank one. Together with the fact that $\K_1(A)$ is an abelian group, we deduce that (*) is also equivalent to
\[
\left\{\begin{array}{ll}
[p]=[q]\text{ in} \Cu(A)\\
\,[u+(1-p)]=[v+(1-p)]\text{ in } \K_1(A)
	\end{array}\right.
\] 
Thus, we have arrived to the desired equivalence. Combined with the exhaustiveness of the maps $\iota$ and $\gamma$, we can (explicitly) construct two monoid morphisms that are inverses of one another, as follows: 
\[
\xymatrix{
\Cu_1(A)_c\ar@{>>}[dr]_{\gamma}\ar@{>>}[rr]^{\iota} &&\iota(\Cu_1(A)_{c})\ar@<0,5ex>@{-->}[dl]^\simeq\\
& {\K_*(A)}_+\ar@<0,5ex>@{-->}[ur]_{} &
}
\]
As of now, using the fact that any element of an ordered group can be uniquely decomposed as the sum of elements of its positive cone (or rather, a difference), we can extend the isomorphism obtain between the positive cones to a (natural) ordered group isomorphism $(K_*(A),K_*(A)_+,[1_A])\simeq (\Gr(\Cu_1(A)_c),\iota(\Cu_1(A)_c),[1_A])$. Therefore, we conclude that the functor $H_*$ yields a natural isomorphism $\eta_*:H_*\circ\Cu_{1,u}\simeq \K_*$.

$\bullet\,$[A faithful restriction of $H_*$] The last step is to show that the restriction of $H_*$ to the (full) subcategory $\Cu^\sim_{u,\alg,w}$ of algebraic $\Cu^\sim$-semigroups with compact order unit satisfying weak cancellation yields a faithful functor. Let us consider two $\Cu^\sim_u$-morphisms $\alpha,\beta:(S,u)\longrightarrow (T,v)$ between two algebraic $\Cu^\sim_u$-semigroups and we assume that $H_*(\alpha)=H_*(\beta)$. In particular, for any $x\in S_c$ we have $[(\alpha_c(x),0)]=[(\beta_c(x),0)]$ in $\iota(T_c)$. Equivalently, for any $x\in S_c$ there exists $t\in S_c$ such that $\alpha_c(x)+t=\beta_c(x)+t$. Using weak cancellation, we obtain that $\alpha_c=\beta_c$ and since any element in $S$ is the supremum of an increasing sequence of compact elements, we deduce that $\alpha=\beta$. \\

Finally, it can be shown that the unitary Cuntz semigroup of any $\CatCa$-algebra of stable rank one and real rank zero satisfies weak cancellation. (See \cite[Proposition 2.5]{AL22}.) For the sake of completeness, we mention that the proof entirely relies on the following fact: for a $\CatCa$-algebra of real rank zero (and stable rank one) and any $I,J\in\Lat(A)$ such that $I\subseteq J$, the induced $\K_1$-morphism $\delta_{IJ}$ is also injective.

Putting all the pieces together, we conclude that when restricting to the (full) subcategory of $\CatCa$-algebras of stable rank one and real rank zero, we can fully recover $\K_*:\CatCa_{\sr1,\rr0}\longrightarrow \AbGp_u$ from $\Cu_{1,u}:\CatCa_{\sr1,\rr0}\longrightarrow \Cu^\sim_{u,alg,w}$ through $H_*$. Let us mention that the converse is also true and proven in \cite[Theorem 2.11]{AL22}.

Lastly, we restate the notable (and correct) classifications results by means of $\K_*$ that will lead to the corrected version of \cite[Corollary 5.21]{C21}.

\begin{thm*}\textnormal{(}\cite[Corollary 4.9]{EG96} - \cite[Theorem 3.21]{EGS98}, \cite[Theorem 7.3 - Theorem 7.4]{E93}\textnormal{)}

(i) The functor $\K_*$ is a complete invariant for simple unital $AH_{d}$ algebras of real rank zero.

(ii) Let $A,B$ be (unital) $\A\!\T$ algebras of real rank zero and let $\alpha:\K_*(A)\longrightarrow \K_*(B)$ be a scaled ordered group morphism. Then there exists a unique $^*$-homomorphism (up to approximate unitary equivalence) $\phi:A\longrightarrow B$ such that $\K_*(\phi)=\alpha$.
\end{thm*}

\begin{cor*}
(i) The functor $\Cu_{1,u}$ is a complete invariant for simple unital $AH_{d}$ algebras of real rank zero.

(ii) The functor $\Cu_{1,u}$ classifies homomorphisms of unital $\A\!\T$ algebras of real rank zero.
\end{cor*}


\begin{thebibliography}{99}
\bibitem{AL22} Q. An and Z. Liu, \emph{A total Cuntz semigroup for $\CatCa$-algebras of stable rank one}. Preprint. Available at \url{https://arxiv.org/abs/2201.10763}.
\bibitem{ADPS14} R. Antoine, M. Dadarlat, F. Perera, and L. Santiago, \emph{Recovering the Elliott invariant from the Cuntz semigroup}, Trans. Amer. Math. Soc. 366, (2014), no. 6, pp. 2907-2922.
\bibitem{APS11} R. Antoine, F. Perera, and L. Santiago, \emph{Pullbacks, $\mathcal{C}(X)$-algebras, and their Cuntz semigroups}, J. Funct. Anal. 260, (2011), no. 10, pp. 2844-2880. 
\bibitem{APT14} R. Antoine, F. Perera, and H. Thiel, \emph{Tensor products and regularity properties of Cuntz semigroups}, Mem. Amer. Math. Soc. 251, (2018), viii+191.
\bibitem{APRT21} R. Antoine, F. Perera, L. Robert, and H. Thiel, \emph{C*-algebras of stable rank one and their Cuntz semigroups}, Duke Math. J. (2020). To appear.
\bibitem{APT09} P. Ara, F. Perera, and A. S. Toms, \emph{K-theory for operator algebras. Classification of C*-algebras, Aspects of operator algebras and applications}, Contemp. Math., vol. 534, Amer. Math. Soc., Providence, RI, 2011, pp. 1-71.
\bibitem{Bl77} B. Blackadar, \emph{Infinite tensor products of C*-algebras}, Pacific J. Math. 72, (1977), no. 2, pp. 313-334.
\bibitem{Bl86} B. Blackadar, \emph{K-Theory for Operator Algebras}, Springer-Verlag, New York, 1986.
\bibitem{Br77} L. G. Brown, \emph{Stable isomorphism of hereditary subalgebras of C*-algebras}, Pac. J. Math. 71, (1977), pp. 335-348.
\bibitem{C21} L. Cantier, \emph{A unitary Cuntz semigroup for $\CatCa$-algebras of stable rank one}, J. Funct. Anal. 281, (2021), no. 9, 109175.
\bibitem{C21b} L. Cantier, \emph{Unitary Cuntz semigroups of ideals and quotients}, M{\"u}nster J. of Math. 14, (2021), pp. 585-606.
\bibitem{C22} L. Cantier, \emph{Uniformly based Cuntz semigroups and approximate intertwinings}, Int. J. of Mathematics, (2022), Volume No. 33, Issue No. 09, 2250062.
\bibitem{CEI08} K. T. Coward, G. A. Elliott, and C. Ivanescu, \emph{The Cuntz semigroup as an invariant for C*-algebras}, J. Reine Angew. Math. 623, (2008), pp. 161-193. 
\bibitem{C78} J. Cuntz, \emph{Dimension functions on simple C*-algebras}, Math. Ann. 233, (1978), pp. 145-153.
\bibitem{E93} G. A. Elliott, \emph{On the classification of $\CatCa$-algebras of real rank zero}, J. Reine Angew. Math. 443, (1993), pp. 179-219.
\bibitem{EG96} G. A. Elliott and G. Gong, \emph{On the classification of $\CatCa$-algebras of real rank zero, II}, Ann. of Math. 144, (1996), no. 3., pp. 497-610.
\bibitem{EGLN21} G. A. Elliott, G. Gong, H. Lin and Z. Niu, \emph{On the classification of simple amenable C*-algebras with finite decomposition rank, II}. Preprint. Available at \url{arXiv:1507.03437}.
\bibitem{EGS98} G. A. Elliott, G. Gong and H. Su, \emph{On the classification of C*-algebras of real rank zero, IV: Reduction to local spectrum of dimension two}, Operator Algebras and Their Applications, II, Fields Inst. Commun., vol. 20, Amer. Math. Soc., (1998), pp. 73-95.
\bibitem{EK91} D. Evans and A. Kishimoto \emph{Compact group actions on UHF algebras obtained by folding the interval}, J. Funct. Anal., Volume 98, Issue 2, 1991, pp. 346-360.
\bibitem{GJL19} G. Gong, C. Jiang and L. Li, \emph{Hausdorffified algebraic $K_1$-group and invariants for C*- algebras with the ideal property}, Ann. K-Theory, vol. 5 no. 1, (2020), pp. 43-78.
\bibitem{GJL20}  G. Gong, C. Jiang and L. Li, \emph{A classification of inductive limit C*-algebras with ideal property}, (2020). To appear.
\bibitem{GLN1} G. Gong, H. Lin, and Z. Niu, \emph{A classification of finite simple amenable $\mathcal{Z}$-stable C*-algebras, I: C*-algebras with generalized tracial rank one}, C. R. Math. Acad. Sci. Soc. R. Canada, 42 (2020), pp. 63-450.
\bibitem{GLN2} G. Gong, H. Lin and Z. Niu, \emph{A classification of finite simple amenable $\mathcal{Z}$-stable C*-algebras, II: C*-algebras with rational generalized tracial rank one}, C. R. Math. Acad. Sci. Soc. R. Canada, 42 (2020), pp. 451-539.
\bibitem{G92} K. R. Goodearl, \emph{Notes on a class of simple C*-algebras with real rank zero}, Publ. Mat. (Barcelona) 36, (1992), pp. 637-654. 
\bibitem{JJ11} K. Ji and C. Jiang, \emph{A complete classification of $\AI$  algebra with ideal property}, Canadian. J. Math, 63 (2), (2011), pp. 381-412.
\bibitem{RR07} G. Restorff and E. Ruiz, \emph{On R$\o$rdam's classification of certain C*-algebras with one non-trivial ideal, II}, Math. Scand. 101, (2007), pp. 280-292.
\bibitem{R12} L. Robert, \emph{Classification of inductive limits of one-dimensional NCCW complexes}, Adv. Math. 231, (2012), no. 5, pp. 2802-2836.
\bibitem{RS19} L. Robert and L. Santiago. \emph{A revised augmented Cuntz semigroup}. Math. Scand., (2019). To appear. 
\bibitem{R97} M. R$\o$rdam, \emph{Classification of extensions of certain C*-algebras by their six term exact sequences in K-theory}, Math. Ann. 308, (1997), no. 1, pp. 93-117.
\bibitem{T19} H. Thiel, \emph{The Cuntz Semigroup}, Lecture notes from a course at the University of M\"unster, winter semester 2016-17. Available at: \url{https://ivv5hpp.uni-muenster.de/u/h_thie08/teaching/CuScript.pdf}.
\bibitem{TV21} H. Thiel, and E. Vilalta, \emph{Covering dimension of Cuntz semigroups}, Adv. Math., (2021). To appear.
\bibitem{Th94} K. Thomsen, \emph{Inductive Limits of Interval Algebras: The Tracial State Space}, Amer. J. Math., Vol. 116, No. 3, (1994), pp. 605-620.
\bibitem{T08} A. S. Toms, \emph{On the classification problem for nuclear C*-algebras}, Ann. of Math. (2) 167, (2008), no. 3, pp. 1029-1044.
\end{thebibliography}
\end{document}